%% file: exact_deviations.tex
\documentclass[final, twoside, a4paper]{amsart}
\usepackage[latin1]{inputenc}
\usepackage[T1]{fontenc}
\usepackage{textcomp}
\usepackage[english]{babel}
\usepackage{amsmath}
\usepackage{amsfonts}
\usepackage{amssymb}
\usepackage{amsopn}
\usepackage{amscd}
\usepackage{amsthm}
\usepackage{mathrsfs}
\usepackage{graphicx}
\usepackage{enumerate}
\theoremstyle{plain}
\newtheorem{theorem}{Theorem}[section]
\newtheorem{lemma}[theorem]{Lemma}
\newtheorem{prop}[theorem]{Proposition}
\newtheorem{corollary}[theorem]{Corollary}
\newtheorem{definition}[theorem]{Definition}
\newtheorem{example}[theorem]{Example}

\theoremstyle{definition}

\theoremstyle{remark}
\newtheorem*{remark}{Remark}

\newcommand{\R}{\mathbb{R}}

\newcommand{\N}{\mathbb{N}}

\newcommand{\Pro}{\mathcal{P}}
\newcommand{\pr}{\mathbb{P}}

\newcommand{\bt}{\mathbf{T}}

\newcommand{\eps}{\varepsilon}

\title[Simple bounds in 1-Wasserstein distance]{Simple bounds for the convergence of empirical and occupation measures in 1-Wasserstein distance}

\author{Emmanuel Boissard}

\address{Universit\'{e} Paul Sabatier, Toulouse, France}

\begin{document}

\date{\today}

\begin{abstract}
 We study the problem
 of non-asymptotic deviations between a reference measure $\mu$
and its empirical version $L_n$, in the $1$-Wasserstein metric, under the standing assumption that
$\mu$ satisfies a transport-entropy inequality. We extend some results of F. Bolley, A. Guillin and C. Villani \cite{quantit_conc_ineq}
with simple proofs. Our methods are based on concentration inequalities and extend to the general setting of measures
 on a Polish space.
Deviation bounds for the occupation measure of a contracting Markov chain in $W_1$ distance are also given.

Throughout the text, several examples are worked out, including the cases of Gaussian measures on separable Banach spaces, 
and laws of diffusion processes.

\end{abstract}

\maketitle

\tableofcontents

\input{introduction}

\input{main_results}

\input{proof_of_main_thm}

\input{variant}

\input{dependent_case}

\appendix
\input{appendix_transportation_inequalities}
\input{appendix_1}
\input{appendix_2}
\input{appendix_3}

\bibliography{exact_deviations}{}
\bibliographystyle{plain}

\end{document}

%% file: introduction.tex
\section{Introduction}

\subsection{Generalities}

In the whole paper, $(E, d)$ will denote a Polish space with metric $d$, equipped with its Borel $\sigma$-field and $\Pro(E)$ will denote
the set of probability measures over $E$.
Consider $\mu \in \Pro(E)$ and 
 a sequence of i.i.d. variables $X_i$, $1 \leq i \leq n$, with common law $\mu$. Let 

\begin{equation}
L_n = \frac{1}{n} \sum_{i = 1}^{n} \delta_{X_i} 
\end{equation}

 denote the empirical measure
associated with the i.i.d. sample $(X_i)_{1 \leq i \leq n}$, then with probability 1, $L_n \rightharpoonup \mu$ as $n \rightarrow + \infty$  (here 
the arrow denotes narrow convergence, or convergence against all 
bounded continuous functions over $E$). This theorem is known as the empirical law of large number or Glivenko-Cantelli theorem and is due in this form to Varadarajan
\cite{varadarajan1958convergence}.
Quantifying the speed of convergence for an appropriate notion of distance between probability measures is an old problem, with notable importance in statistics.
For many examples, we refer to the book of Van der Vaart and Wellner \cite{van1996weak} and the Saint-Flour course of P.Massart \cite{massart1896concentration}.

Our aim here is to 
study non-asymptotic deviations in $1$-Wasserstein distance. This is a problem of interest in the fields of
statistics and numerical probability.
 More specifically, we provide bounds for the quantity
$\pr( W_1( L_n, \mu) \geq t)$
for $t > 0$, i.e. we quantify the speed of convergence of the variable $W_1( L_n, \mu)$ to $0$ in probability.

This paper seeks to complement the work of F.Bolley, A.Guillin and C.Villani in \cite{quantit_conc_ineq}
where such estimates are obtained for measures supported in $\R^d$. We sum up (part of) their result here.
 Suppose that $\mu$ is a probability measure on
$\R^d$ for $1 \leq p \leq 2$ that satisfies a $\bt_p(C)$ transportation-entropy inequality, that is

\begin{equation*}
 W_p(\nu, \mu) \leq \sqrt{C H(\nu | \mu)} \text{ for all } \nu \in \Pro_p(\R^d)
\end{equation*}

(see below for definitions).
They obtain a
non-asymptotic Gaussian deviation estimate for the $p-$Wasserstein distance between the empirical and true measures~:

\begin{equation*}
 \pr (W_p(L_n, \mu) \geq t) \leq C(t) \exp ( - K n t^2).
\end{equation*}

This is an effective result : the constants $K$ and $C(t)$ may be explicitely computed
 from the value of some square-exponential moment of $\mu$ and the constant $C$ appearing in the transportation inequality.

The strategy used in \cite{quantit_conc_ineq} relies on a non-asymptotic version of (the upper bound in) Sanov's theorem.
Roughly speaking, Sanov's theorem states that the proper rate function for the deviations of empirical measures is the entropy functional,
or in other words that for 'good' subsets $A \in \Pro(E)$,

\begin{equation*}
 \pr (L_n \in A) \asymp e^{- n H(A | \mu)}
\end{equation*}

where $H(A | \mu) = \inf_{\nu \in A} H( \nu | \mu)$ (see \cite{dembo1993large} for a full statement of the theorem).

In a companion work \cite{boissard2011bounding}, we derive sharper bounds for this problem, using a construction originally due to R.M. Dudley \cite{dudley1969speed}.
The interested reader may refer to \cite{boissard2011bounding} for a summary of existing results.
Here, our purpose is to show that in the case $p = 1$, the results of \cite{quantit_conc_ineq} can be recovered with simple
arguments of measure concentration, and to give various extensions of interest.

\begin{itemize}
 \item We would like to consider spaces more general than $\R^d$.
 \item We would like to encompass a wide class of measures in a synthetic treatment. In order to do so we will consider more general transportation inequalities, see below.
 \item Another interesting feature is to extend the result to dependent sequences such as the occupation measure of a Markov chain.
This is a particularly desirable feature in applications : one may wish to approximate a distribution that is unknown, or from which it is
practically impossible to sample uniformly, but that is known to be
the invariant measure of a simulable Markov chain.
\end{itemize}

\emph{Acknowledgements.} The author thanks his advisor Patrick Cattiaux for suggesting the problem and for his advice. Arnaud Guillin is also thanked for enriching conversations.

\bigskip

In the remainder of this section, we introduce the tools necessary in our framework : transportation distances and transportation-entropy inequalities.
In Section \ref{results}, we give our main results, as well as explicit estimates in several relevant cases. Section \ref{proof_of_main_thm} is devoted
to the proof of the main result. Section \ref{section_variant} is devoted to the proof of Theorem \ref{variant}. 
In Section \ref{dependent_case} we show how our strategy of proof can extend to the dependent case.

\subsection{A short introduction to transportation inequalities}

\subsubsection{Transportation costs and Wasserstein distances}

We recall here basic definitions and propositions ; for proofs and a thorough account of this rich theory,
 the reader may refer to \cite{optimal_transport_villani}.
 Define $\Pro_p$, $ 1 \leq p < + \infty$, as the set of probability measures with a
finite $p$-th moment. The $p$-Wasserstein metric $W_p(\mu, \nu)$ between $\mu, \nu \in \Pro_p$ is defined by

\begin{equation*}
W_p^p(\mu, \nu) =  \text{inf} \int d^p(x,y) \pi(dx, dy)
\end{equation*}

where the infimum is on probability measures $\pi \in \Pro(E \times E)$ with marginals $\mu$ and $\nu$.
The topology induced by this metric is slightly stronger than the weak topology : namely, 
convergence of a sequence $(\mu_n)_{n \in \N}$ to a measure $\mu \in \Pro_p$ in the $p$-Wasserstein metric is equivalent to
the weak convergence of the sequence plus a uniform bound on the $p$-th moments of the measures $\mu_n$, $n \in \N$.

We also recall the well-known Kantorovich-Rubinstein dual characterization of $W_1$ : let $\mathcal{F}$ denote the set of $1$-Lipschitz functions
$f : E \rightarrow \R$ that vanish at some fixed point $x_0$. We have :

\begin{equation} \label{kantorovich_rubinstein}
 W_1(\mu, \nu) = \inf_{f \in \mathcal{F}} \int f d \mu - \int f d \nu.
\end{equation}

\subsubsection{Transportation-entropy inequalities}

For a very complete overview of the subject, the reader is invited to consult the review \cite{gozlan2010transport}.
More facts and criteria are gathered in Appendix \ref{appendix_transport_inequalities}
For $\mu, \nu \in \Pro(E)$, define the relative entropy $H(\nu | \mu)$ as

\begin{equation*}
 H( \nu | \mu) = \int_E \log \frac{d \nu}{d \mu} \nu( d x)
\end{equation*}

if $\nu$ is absolutely continuous relatively to $\mu$, and $H( \nu | \mu) = + \infty$ otherwise.
Let $\alpha : [0, \ + \infty) \rightarrow \R$ denote a convex, increasing, left-continous function such that $\alpha(0) = 0$.

\begin{definition}

We say that $\mu \in \Pro_p(E)$ satisfies a $\bt_p(C)$ inequality 
for some $C > 0$ if
for all $\nu \in \Pro_p(E)$,

\begin{equation}
 W_p(\mu, \nu) \leq \sqrt{C H( \nu | \mu) }.
\end{equation}

 We say that $\mu \in \Pro(E)$ satisfies a  $\alpha(\mathcal{T}_d)$ inequality if for all $\nu \in \Pro(E)$,

\begin{equation} \label{t_c_eq}
 \alpha( W_1(\mu, \nu)) \leq H(\nu | \mu).
\end{equation}

\end{definition}

 Observe that $\bt_1(C)$ inequalities are particular cases of $\alpha(\mathcal{T}_d)$ inequalities with
 $\alpha(t) = \frac{1}{C} t^{2/p}$.
From here on, our focus will be on $\alpha(\mathcal{T}_d)$ inequalities.

%% file: main_results.tex
\section{Results and applications} \label{results}

\subsection{General bounds in the independent case}

Let us first introduce some notation :
if $K \subset E$ is compact and $x_0 \in K$, we define the set $\mathcal{F}_K$ of $1$-Lipschitz functions over $K$ vanishing at $x_0$, which is
is also compact w.r.t. the uniform distance (as a consequence of the Ascoli-Arzela theorem). We will also need the following definition~:

\begin{definition}
 Let $(A, d)$ be a totally bounded metric space. For every $\delta > 0$, define the \textit{covering number}
$\mathcal{N}(A, \delta)$  of order $\delta$ for $A$
as the minimal number of balls of radius $\delta$ needed to cover $A$.
\end{definition}

We state our first result in a fairly general fashion.

\begin{theorem} \label{main_thm}
 Suppose that $\mu \in \Pro(E)$ satisfies a $\alpha(\mathcal{T}_d)$ inequality. Let $a > 0$ be such that
$E_{a, 1} = \int e^{a d(x_0, x)} \mu (dx) \leq 2$.
Choose a compact $K \subset E$ such that

\begin{equation*}
\mu(K^c) \leq \left[ \frac{32}{a t} \log \frac{32}{a t} - \frac{32}{a t} + 1 \right]^{-1}.
\end{equation*}

Denote 

\begin{equation}
\mathcal{C}_t = \mathcal{N}(\mathcal{F}_K, t/8). 
\end{equation}
 
We have

\begin{equation} \label{main_eq_bound}
\pr( W_1(L_n, \mu) \geq  t) \leq \exp -n \alpha \left[ t/2 -\Gamma(\mathcal{C}_t, n) \right]
\end{equation}

where $\Gamma(\mathcal{C}_t, n) = \inf_{\lambda > 0} 1/\lambda[ \log \mathcal{C}_t + n \alpha^\circledast (\lambda/n) ]$,
and with the convention that $\alpha(x) = 0$ for $x < 0$.

\end{theorem}

\begin{remark}
 With a mild change in the proof, one may replace in (\ref{main_eq_bound}) the term $t/2$ by $ct$ for any $c < 1$,
with the trade-off of choosing a larger compact set, and thus a larger value of $\mathcal{C}_t$. For the sake of readability we do not
make further mention of this.
\end{remark}

The result in its general form is abtruse, but it yields interesting results as soon as one knows more about $\alpha$. Let us give a few examples.

\begin{corollary} \label{corollary_bt}
 If $\mu$ satisfies $\bt_1(C)$, we have

\begin{equation*}
 \pr( W_1(L_n, \mu) \geq  t) \leq \mathcal{C}_t \exp - \frac{1}{8 C} n t^2.
\end{equation*}

\end{corollary}

\begin{corollary} \label{corollary_2}
 Suppose that $\mu$ verifies the modified transport inequality

\begin{equation*}
 W_1(\nu, \mu) \leq C \left( H(\nu | \mu) + \sqrt{H(\nu | \mu)} \right)
\end{equation*}

(as observed in paragraph \ref{integral_criteria}, this is equivalent to the finiteness of an exponential moment for $\mu$). Then, for $t \leq C/2$,

\begin{equation*}
\pr( W_1(L_n, \mu) \geq  t) \leq A(n, t) \exp - \frac{(\sqrt{2} - 1)^2}{2 C^2} n t^2
\end{equation*}

where

\begin{equation*}
 A(n, t) = \exp  \displaystyle \left[ 4 (\sqrt{2} - 1)^2 n (\sqrt{1 + \frac{n}{\log \mathcal{C}_t}} -1)^{-2} \right]
\end{equation*}

(observe that $A(n, t) \rightarrow \mathcal{C}_t$ when $n \rightarrow + \infty$).

\end{corollary}

\begin{proof}[Proof of Corollary \ref{corollary_bt}]
 In this case, we have $\alpha(t) = \frac{1}{C}t^2$, and so 

\begin{equation*}
\Gamma(\mathcal{C}_t, n) = \sqrt{\frac{C \log \mathcal{C}_t}{n}}, 
\end{equation*} 

so that we get

\begin{equation*}
 \pr( W_1(L_n, \mu) \geq  t) \leq \exp -\frac{n}{C} (\frac{t}{2} - \sqrt{\frac{C \log \mathcal{C}_t}{n}})^2
\end{equation*}

and conclude with the elementary inequality $(a - b)^2 \geq \frac{1}{2}a^2 - b^2$.

\end{proof}

\begin{proof}[Proof of Corollary \ref{corollary_2}]
 
Here, $\alpha(x) = \frac{1}{4}( \sqrt{1 + \frac{4 x}{C}} - 1)^2$, and one can get the bound 

\begin{equation*}
\Gamma(\mathcal{C}_t, N) \leq \frac{C}{\sqrt{1 + \frac{N}{\log \mathcal{C}_t}} - 1}.
\end{equation*}

By concavity of the square root function, for $u \leq 1$, we have $\sqrt{1 + u} - 1 \geq (\sqrt{2}-1)u$. Thus, for $t \leq \frac{C}{2}$, we have

\begin{eqnarray*}
 \alpha( \frac{t}{2} - \Gamma(\mathcal{C}_t, N) ) & \geq & \frac{(\sqrt{2} - 1)^2}{4} ( \frac{2}{C}t - \frac{4}{\sqrt{1 + \frac{N}{\log \mathcal{C}_t}} -1})^2 \\
{} & \geq & \frac{(\sqrt{2} - 1)^2}{2 C^2} t^2 - 4 (\sqrt{2} - 1)^2(\sqrt{1 + \frac{N}{\log \mathcal{C}_t}} -1)^{-2}
\end{eqnarray*}

(in the last line we have used again the inequality $(a-b)^2 \geq \frac{a^2}{2} - b^2$). This in turn gives the announced result.

\end{proof}

Our technique of proof, though related to the one in \cite{quantit_conc_ineq},
is based on different arguments : we make use of   
the tensorization properties of transportation inequalities
as well as the estimates (\ref{cond_laplace}) in the spirit of Bobkov-G\"{o}tze, 
instead of a Sanov-type bound. 
The notion that is key here is the phenomenon of concentration of measure (see e.g. \cite{ledoux2001concentration}) : its relevance in statistics was put forth
very explicitely in \cite{massart1896concentration}. 
We may sum up our approach as follows :
first, we rely on existing tensorization results to obtain concentration of $W_1(L_n, \mu)$
around its mean $\mathbb{E}[W_1(L_n, \mu)]$, and in a second time we estimate the decay of
the mean as $n \rightarrow + \infty$. Despite technical difficulties, the arguments are mostly elementary.

The next theorem is a variation on Corollary \ref{corollary_bt}. Its proof is based on different arguments, and it is postponed to Section \ref{section_variant}.
We will use this theorem to obtain bounds for Gaussian measures in Theorem \ref{variant}.

\begin{theorem} \label{variant}
 
Let $\mu \in \Pro(E)$ satisfy a $\bt_1(C)$ inequality. Then :

\begin{equation*}
 \pr( W_1( \mu, L_n) \geq t) \leq K_t e^{-nt^2/8C}
\end{equation*}

where

\begin{equation*}
 K_t =  \exp \left[ \frac{1}{C} \inf_\nu \text{Card }(\text{Supp } \nu) (\text{Diam Supp } \nu)^2 \right]
\end{equation*}

and $\nu$ runs over all probability measures with finite support such that $W_1( \mu, \nu ) \leq t/4$.
\end{theorem}

\begin{remark}
 As earlier, we could improve the factor $1/8C$ in the statement above to any constant $c < 1/C$, with the trade-off of a larger constant $K_t$.
\end{remark}

\subsection{Comments}

We give some comments on the pertinence of the results above. First of all, we argue that the asymptotic order of magnitude
of our estimates is the correct one. The term ``asymptotic'' here means
that we consider the regime $n \rightarrow + \infty$, and the relevant tool in this setting is Sanov's large deviation principle
for empirical measures.
A technical point needs to be stressed : there are several variations of Sanov's theorem, and the most common ones (see e.g. \cite{dembo1993large})
deal with the weak topology on probability measures. What we require is a version of the principle that holds for
 the stronger topology induced by the $1$-Wasserstein metric, which leads to 
slightly more stringent 
assumptions on the measure than in Theorem \ref{main_thm}. With this in mind, we quote the following result from Wang \cite{wang2010sanov} :

\begin{prop}
 
Suppose that $\mu \in \Pro(E)$ satisfies 
$\int e^{a d(x, x_0)} \mu(dx) < + \infty$
 for all  $a > 0$ and some $x_0 \in E$,
and  a $\alpha(\mathcal{T}_d)$ inequality.
Then :

\begin{itemize}
 \item for all $A \subset \Pro(E)$ closed for the $W_1$ topology,
\begin{equation*}
 \limsup_{n \rightarrow + \infty} \frac{1}{n} \log \mu(A) \leq - \inf_{ \nu \in A} H(\nu | \mu)
\end{equation*}
\item for all $B \subset \Pro(E)$ open w.r.t. the $W_1$ topology,
\begin{equation*}
 \liminf_{n \rightarrow + \infty} \frac{1}{n} \log \mu(B) \geq - \inf_{\nu \in B} H(\nu | \mu).
\end{equation*}

\end{itemize}
\end{prop}

Consider the closed set $A = \{ \nu \in \Pro(E), \, W_1(\mu, \nu) \geq t \}$, then we have according to the above

\begin{equation*}
\limsup_{n \rightarrow + \infty} \frac{1}{n} \log \pr(W_1(L_n, \mu) \geq t)  \leq - \alpha(t).
\end{equation*}

With Theorem \ref{main_thm} (and the remark following it), we obtain the bound

\begin{equation*}
\limsup_{n \rightarrow + \infty} \frac{1}{n} \log \pr(W_1(L_n, \mu) \geq t)  \leq - \alpha(ct) 
\end{equation*}

 for all $c < 1$, and since $\alpha$
is left-continuous, we indeed obtain the same asymptotic bound as from Sanov's theorem.

\bigskip

Let us come back to the non-asymptotic regime.
When we assume for example a $\bt_1$ inequality, we get a bound
 in the form $\pr ( W_1( L_n, \mu) \geq t ) \leq C(t) e^{- C n t ^2}$ involving 
the large constant $C(t)$. By the Kantorovich-Rubinstein dual formulation of $W_1$, this amounts to 
simultaneous deviation
inequalities for all $1$-Lipschitz observables. We recall briefly the well-known fact
 that it is fairly easy to obtain a deviation inequality for one
Lipschitz observable without a constant depending on the deviation scale $t$. 
Indeed, consider a $1$-Lipschitz function $f$ and a sequence
$X_i$ of i.i.d. variables with law $\mu$. By Chebyshev's bound, for $\theta > 0$,

\begin{eqnarray*}
 \pr ( \frac{1}{n} \sum f(X_i) - \int f \mu \geq \eps) & \leq & \exp - n [ \theta \eps - \log (\int e^{\theta f(x)} \mu (dx) e^{- \theta \int f \mu}) ]
\end{eqnarray*}

According to Bobkov-G\"{o}tze's dual characterization of $\bt_1$, the term inside the $\log$ is bounded above by $e^ {C \theta^2}$, for some positive $C$, whence
$\pr ( \frac{1}{n} \sum f(X_i) - \int f \mu \geq \eps) \leq \exp - n [ \theta \eps - C \theta^2 ]$.
Finally, take $\theta = \frac{1}{2 C} \eps$ to get

\begin{equation*}
\pr ( \frac{1}{n} \sum f(X_i) - \int f \mu \geq \eps) \leq e^{- C n t^2/2}. 
\end{equation*}

Thus, we may see the multiplicative constant that we obtain as a trade-off
for the obtention of uniform deviation estimates on all Lipschitz observables.

\bigskip

\subsection{Examples of application}

 For practical purposes, it is important to give the order of magnitude
 of the multiplicative constant $\mathcal{C}_t$ depending on $t$.
We address this question on several important examples in this paragraph.

\subsubsection{The $\R^d$ case}

\begin{example} \label{result_R}
 
Denote $\theta(x) = 32 x \log \left[ 2 \left( 32 x \log 32 x - 32 x + 1  \right) \right]$.
In the case $E = \R^d$, the numerical constant $\mathcal{C}_t$ appearing in Theorem \ref{main_thm}
satisfies :

\begin{equation} \label{real_case}
 \mathcal{C}_t \leq 2 \left( 1 + \theta ( \frac{1}{a t} ) \right) 2^{ \displaystyle C_d \theta( \frac{1}{a t} )^d }
\end{equation}

where $C_d$ only depends on $d$. In particular, for all $t \leq \frac{1}{2 a}$, there exist numerical constants $C_1$ 
and $C_2$ such that

\begin{equation*}
 \mathcal{C}_t \leq C_1 (1 + \frac{1}{at} \log \frac{1}{at}) e^{ \displaystyle C_d C_2^d (\frac{1}{at} \log \frac{1}{a t})^d}.
\end{equation*}

\end{example}

\begin{remark}
 The constants $C_d$, $C_1$, $C_2$ may be explicitely determined from the proof. We do not do so and only state that $C_d$ grows exponentially with $d$.
\end{remark}

\begin{proof}

For a measure $\mu \in \Pro(\R^d)$, a convenient natural choice for a compact set of large measure is a
Euclidean ball. Denote $B_R = \{x \in \R^d, |x| \leq R \}$. We will denote by $C_d$ a constant depending only on
the dimension $d$, that may change from line to line. Suppose that
$\mu$ satisfies the assumptions in Theorem \ref{main_thm}. By Chebyshev's bound,
 $\mu(B_R^c) \leq 2 e^{-a R}$, so we may choose $K = B_{R_t}$ with

\begin{equation*} 
 R_t \geq \frac{1}{a} \log \left[ 2 \left(\frac{32}{a t} \log \frac{32}{ a t} - \frac{32}{a t} + 1 \right) \right] .
\end{equation*}

Next, the covering numbers for $B_R$ are bounded by :

\begin{equation*}
 \mathcal{N}(B_R, \delta) \leq C_d \left( \frac{R}{\delta} \right)^d.
\end{equation*}
Using the bound (\ref{covering_number_lip_2}) of Proposition \ref{prop_covering_number_lip}, we have

\begin{equation*}
 \mathcal{C}_t \leq \left( 2 + 2 \lfloor \frac{32 R_t}{t} \rfloor \right) 2^{ \displaystyle C_d \left(\frac{ 32 R_t}{t} \right)^d }.
\end{equation*}

This concludes the proof for the first part of the proposition.
 The second claim derives from the fact that for $x > 2$, there exists
a numerical constant $k$ such that
$\theta(x) \leq k x \log x$.

\end{proof}

Example \ref{result_R} improves slightly upon the result for the $W_1$ metric in \cite{quantit_conc_ineq}.
One may wonder whether this order of magnitude is close to optimality. It is in fact not sharp, and we point out
where better results may be found.

In the case $d = 1$, 
$W_1(L_n,  \mu)$ is bounded above by the Kolmogorov-Smirnov divergence
$\sup_{x \in \R} | F_n(x) - F(x)|$ where $F_n$ and $F$ denote respectively the cumulative distribution functions (c.d.f.)
of $L_n$ and $\mu$. As a consequence of the celebrated Dvorestky-Kiefer-Wolfowitz theorem
(see \cite{massart1990tight}, \cite{van1996weak}), we have the following :
if $\mu \in \Pro(\R)$ has a continuous c.d.f., then

\begin{equation*}
 \pr(W_1( L_n, \mu) > t ) \leq 2 e^{-2 n t^2}.
\end{equation*}

The behaviour of the Wasserstein distance between empirical and true distribution in one dimension has been very
thoroughly studied by del Barrio, Gin\'{e}, Matran, see \cite{del1999central}.

In dimensions greater than $1$, the result is also not sharp. Integrating (\ref{real_case}), one recovers a bound of the type
$\mathbb{E} (W_1(L_n, \mu)) \leq C n^{-1/(d + 2)} (\log n)^c$. Looking into the proof of our main result, one sees that any improvement of this bound will
automatically give a sharper result than (\ref{real_case}).
For the uniform measure over the unit cube, results have been known for a while. The pioneering work in this framework is
the celebrated article of Ajtai, Komlos and Tusn\'ady \cite{ajtai1984optimal}.
 M.Talagrand \cite{talagrand1992matching} showed that when $\mu$ is the uniform distribution on the unit cube 
(in which case it clearly satisfies a $\bt_1$ inequality) and $d \geq 3$, there exists $c_d \leq C_d$ such that 

\begin{equation*}
 c_d n^{-1/d} \leq \mathbb{E} W_1(L_n, \mu) \leq C_d n^{-1/d}.
\end{equation*}

Sharp results for general measures are much more recent : as a consequence of the results of F. Barthe and C. Bordenave \cite{barthe2011combinatorial},
one may get an estimate of the type $\mathbb{E} W_1(L_n, \mu) \leq c n^{-1/d}$ under some polynomial moment condition on $\mu$.

\subsubsection{A first bound for Standard Brownian motion}

We wish now to illustrate our results on an infinite-dimensional case. A first natural candidate is the law of the standard Brownian motion,
with the sup-norm as reference metric. The natural idea that we put in place in this paragraph is to choose as large compact sets the $\alpha$-H\"{o}lder
balls, which are compact for the sup-norm. However the remainder of this paragraph serves mainly an illustrative purpose : we will obtain
sharper results, valid for general Gaussian measures on (separable) Banach spaces, in paragraph \ref{subsection_gaussian}.

We consider the canonical Wiener space 
$\left( \mathcal{C}([0, 1], \R), \gamma, \| .\|_{\infty} \right)$, 
 where $\gamma$ denotes the 
Wiener measure, under which the coordinate process $B_t : \omega \rightarrow \omega(t)$ is a standard Brownian motion.

\begin{example} \label{example_s_b_m}
 Denote by $\gamma$ the Wiener measure on $\left( \mathcal{C}([0, 1], \R), \gamma, \| .\|_{\infty} \right)$, and
for $\alpha < 1/2$, define

\begin{equation*}
 C_\alpha = 2^{1+ \alpha}\frac{2^{(1-2 \alpha)/4}}{1 - 2^{4/(1-2 \alpha)}} \|Z\|_{4/(1-2\alpha)}
\end{equation*}

where $\|Z\|_p$ denotes the $L_p$ norm of a $\mathcal{N}(0, 1)$ variable $Z$.
There exists $k > 0$ such that for every $t \leq 144/\sqrt{2 \log 2}$, 
$\gamma$ satisfies 

\begin{equation*}
 \pr( W_1( L_n, \gamma) \geq t) \leq \mathcal{C}_t e^{-n t^2/64}
\end{equation*}

with

\begin{equation*}
 \mathcal{C}_t  \leq \exp \exp (k C_\alpha \frac{\sqrt{\log 1/t}}{t})^{1/\alpha}.
\end{equation*}

\end{example}

\begin{proof}

For $0 < \alpha \leq 1$, define the $\alpha$-H\"{o}lder semi-norm as

\begin{equation*}
 | x|_\alpha = \sup_{t,s \in [0, 1]} \frac{|x(t) - x(s)|}{|t-s|^\alpha}.
\end{equation*}

Let $0 < \alpha \leq 1$ and denote by $C_\alpha$ the Banach space of $\alpha$-H\"{o}lder continuous functions vanishing at $0$, endowed with the norm
$\|.\|_\alpha$. 
It is a classical fact that the Wiener measure is
concentrated on $C_\alpha$ for all $\alpha \in ]0, 1/2[$.
By Ascoli-Arzela's theorem, $C_\alpha$ is compactly embedded in $\mathcal{C}([0, 1], \R)$, or in other words
the $\alpha$-H\"{o}lder balls $B_{\alpha, R} = \{ x \in \mathcal{C}([0, 1], \R), \, \|x\|_\alpha \leq R \}$
are totally bounded for the uniform norm. This makes $B(\alpha, R)$ good candidates for compact spaces of large measure.
We need to evaluate how big $B(\alpha, R)$ is w.r.t. $\gamma$.

To this end we use the fact that the Wiener measure is also a Gaussian measure on $C_\alpha$ (see \cite{baldi1992large}).
 Therefore Lemma \ref{lemma_gaussian_tail}
applies : denote 

\begin{equation*}
 m_\alpha = \mathbb{E} \sup_t \| B_t \|_\alpha, \quad s^2_\alpha = \mathbb{E} (\sup_t \| B_t \|_\alpha)^2,
\end{equation*}

we have 

\begin{equation*}
 \gamma ( B(\alpha, R)^c) \leq 2 e^{- (R - m_\alpha)^2/2 s^2_\alpha}
\end{equation*}

for $R \geq m_\alpha$. Choosing

\begin{equation} \label{ineq_requirement}
 R_t \geq m_\alpha + \left[ 2 s^2_\alpha \log 2 ( \frac{32}{at} \log \frac{32}{at} - \frac{32}{at} + 1 )\right]^{1/2}
\end{equation}

guarantees that

\begin{equation*}
\gamma ( B(\alpha, R_t)^c) \leq \left( \frac{32}{at} \log \frac{32}{at} - \frac{32}{at} + 1  \right)^{-1}.
\end{equation*}

On the other hand, according to Corollary \ref{holder_moments_sbm}, $m_\alpha$ and $s_\alpha$ are bounded by $C_\alpha$.
And Lemma \ref{lemma_small_exp_moment_wiener} shows that choosing $a = \sqrt{2 \log 2}/3$ ensures
$\mathbb{E} e^{a \sup_t |B_t| } \leq 2$.

Elementary computations show that for $t \leq 144/\sqrt{2 \log 2}$, we can pick

\begin{equation*}
 R_t = 3 C_\alpha\sqrt{ \log(96/(\sqrt{2 \log 2}t))}
\end{equation*}

to comply with the requirement in (\ref{ineq_requirement}).

Bounds for the covering numbers in $\alpha$-H\"{o}lder balls are computed in \cite{bolley2005quantitative} :
\begin{equation} \label{covering_number_holder}
 \mathcal{N}(B(\alpha, R), \delta) \leq 10 \frac{R}{\delta} \exp \left[\log (3) 5^{\frac{1}{\alpha}} \left(\frac{R}{\delta}\right)^{\frac{1}{\alpha}} \right].
\end{equation}

We recover the (unpretty !) bound

\begin{align*}
\mathcal{C}_t & \leq 2(1 + 96 \frac{C_\alpha}{t} \sqrt{\log 96/(\sqrt{2 \log 2}t)}) \exp \left[ 240 \log 2 \frac{C_\alpha}{t} \sqrt{\log 96/(\sqrt{2 \log 2}t)} \right. \\
{} & \times \left. \exp \log 3 \left(120 \frac{C_\alpha}{t} \sqrt{\log 96/(\sqrt{2 \log 2}t)} \right)^{1/\alpha} \right].
\end{align*}

The final claim in the Proposition is obtained by elementary majorizations.

\end{proof}

\subsubsection{Paths of S.D.E.s}

H.Djellout, A.Guillin and L.Wu established a $\bt_1$ inequality for paths of S.D.E.s that
allows us to work as in the case of Brownian motion. We quote their result from \cite{djellout_guillin_wu}.

Consider the S.D.E. on $\R^d$

\begin{equation} \label{s_d_e_example}
 d X_t = b(X_t) dt + \sigma(X_t) d B_t, \quad X_0 = x_0 \in \R^d
\end{equation}

with $b : \R^d \rightarrow \R^d$, $\sigma : \R^d \rightarrow \mathcal{M}_{d \times m}$ and $(B_t)$ is a 
standard $m$-dimensional Brownian motion. We assume that $b$ and $\sigma$ are locally Lipschitz and that
for all $x, y \in \R^d$,

\begin{equation*}
 \sup_x | \sqrt{ \text{tr}\sigma(x)^t\sigma(x)} | \leq A, \quad \langle y - x, b(y) - b(x) \rangle \leq B(1 + |y-x|^2)
\end{equation*}

For each starting point $x$ it has a unique non-explosive solution denoted $(X_t(x)_{t \geq 0}$ and we denote its law on
$\mathcal{C}([0, 1], \R^d)$ by $\mathbb{P}_x$.

\begin{theorem} [\cite{djellout_guillin_wu}]
 Assume the conditions above. There exists $C$ depending on $A$ and $B$ only such that
for every $x \in \R^d$, $\mathbb{P}_x$ satisfies a $\bt_1(C)$ inequality on the space 
$\mathcal{C}([0, 1], \R^d)$ endowed with the sup-norm.
\end{theorem}

We will now state our result. A word of caution : in order to balance readability, the following computations are neither
optimized nor made fully explicit. However it should be a simple, though dull, task for the reader to track the dependence of
the numerical constants on the parameters.

From now on we make the simplifying assumption that the drift coefficient is globally bounded by $B$
(this assumption is certainly not minimal).

\begin{example}
 Let $\mu$ denote the law of the solution of the S.D.E. (\ref{s_d_e_example}) on
the Banach space $C([0, 1], \R^d)$ endowed with the sup-norm. Let $C$ be such that $\mu$ satisfies
$\bt_1(C)$.
For all $0 < \alpha < 1/2$ there exist $C_\alpha$ and $c$ depending only on
$A$, $B$, $\alpha$ and $d$, and such that for $t \leq c$,

\begin{equation*}
 \pr(W_1(L_n, \mu) \geq t) \leq \mathcal{C}_t e^{- n t^2/ 8 C}
\end{equation*}

and 

\begin{equation*}
 \mathcal{C}_t \leq \exp \exp \left[ C_\alpha \left( \log \frac{1}{t} \right)^{-1 + 1/2\alpha} \left( \frac{1}{t} \right)^{-1 + 3/2\alpha} \right].
\end{equation*}

\end{example}

\begin{proof}

The proof goes along the same lines as the Brownian motion case, so we only outline the important steps.
First, there exists $a$ depending explicitely on $A$, $B$, $d$ such that $\mathbb{E}_{\mathcal{P}_x} e^{a\|X_.\|_\infty} \leq 2$ :
this can be seen by checking that the proof of Djellout-Guillin-Wu actually gives the value of a Gaussian moment for 
$\mu$ as a function of $A$, $B$, $d$, and using standard bounds.

Corollary \ref{corollary_holder_sde} applies for $\alpha < 1/2$ and $p$ such that $1/p = 1/2 - \alpha$ : 
there exists $C' < + \infty$ depending
explicitely on $A$, $B$, $\alpha$, $d$, such that $\mathbb{E} \|X_. \|_{\alpha}^p \leq C'$.
Consequently,

\begin{equation*}
 \mu ( B(\alpha, R)^c) \leq C'/ R^p.
\end{equation*}

So choosing

\begin{equation*}
 R = \left(C'  (\frac{32}{at} \log \frac{32}{at} - \frac{32}{at} + 1) \right)^{1/p}
\end{equation*}

guarantees that

\begin{equation*}
\mu ( B(\alpha, R_t)^c) \leq \left( \frac{32}{at} \log \frac{32}{at} - \frac{32}{at} + 1  \right)^{-1}.
\end{equation*}

For $t \leq c$ small enough, $R \leq C'' \left( \frac{1}{t} \log \frac{1}{t} \right)^{1/p}$
with $c$, $C''$ depending on $A$, $B$, $\alpha$, $d$.
The conclusion is reached again by using estimate (\ref{covering_number_holder}) on the covering numbers of H\"{o}lder balls.

\end{proof}

\subsubsection{Gaussian r.v.s in Banach spaces} \label{subsection_gaussian}

In this paragraph we apply Theorem \ref{variant} to the case where $E$ is a separable Banach space with norm $\|.\|$, 
and $\mu$ is a centered Gaussian random variable with values in $E$, meaning that the image of
$\mu$ by every continuous linear functional $f \in E^*$ is a centered Gaussian variable in $\R$.
The couple $(E, \mu)$ is said to be a Gaussian Banach space.

Let $X$ be a $E$-valued r.v. with law $\mu$, and define the weak variance of $\mu$ as 

\begin{equation*}
 \sigma = \sup_{f \in E^*, \, |f| \leq 1} \left( \mathbb{E} f^2(X) \right)^{1/2}.
\end{equation*}

The small ball function of a Gaussian Banach space $(E, \mu)$ is the function

\begin{equation*}
 \psi(t) = - \log \mu(B(0, t)).
\end{equation*}

We can associate to the couple $(E, \mu)$ their Cameron-Martin Hilbert space $H \subset E$, see e.g. \cite{ledoux1996isoperimetry}
for a reference. It is known that the small ball function has deep links with the covering numbers of the unit ball of $H$,
see e.g. Kuelbs-Li \cite{kuelbs1993metric} and Li-Linde \cite{li1999approximation}, as well as with the 
approximation of $\mu$ by measures with finite support in Wasserstein distance (the quantization or optimal quantization problem), 
see Fehringer's Ph.D. thesis \cite{fehringer2001kodierung}, Dereich-Fehringer-Matoussi-Scheutzow \cite{dereich2003link}, Graf-Luschgy-Pag\`{e}s
\cite{graf2003functional}. It should thus come as no surprise that we can give a bound on the constant $K_t$ depending solely on $\psi$
and $\sigma$.
This is the content of the next example.

\begin{example} \label{theorem_gaussian_vector}
 Let $(E, \mu)$ be a Gaussian Banach space. Denote by $\psi$ its small ball function and by $\sigma$ its weak variance.
Assume that $t$ is such that $\psi(t/16) \geq \log 2$ and $t/\sigma \leq 8 \sqrt{2 \log 2}$.
Then

\begin{equation*}
\pr (W_1(L_n, \mu) \geq t) \leq K_t e^{- n t^2/16 \sigma^2}
\end{equation*}

with

\begin{equation*}
 K_t = \exp \exp \left[ c (\psi(t/32) + \log(\sigma/t)) \right]
\end{equation*}

for some universal constant $c$.

\end{example}

A bound for $c$ may be tracked in the proof.

\begin{proof}

\emph{Step 1. Building an approximating measure of finite support.}
 
Denote by $K$ the unit ball of the Cameron-Martin space associated to $E$ and $\mu$, and by
$B$ the unit ball of $E$.
According to the Gaussian isoperimetric inequality (see \cite{ledoux1996isoperimetry}),
for all $\lambda > 0$ and $\varepsilon > 0$,

\begin{equation*}
\mu( \lambda K + \varepsilon B) \geq \Phi \left( \lambda + \Phi^{-1}( \mu( \varepsilon B)) \right)
\end{equation*}

where $\Phi(t) = \int_{- \infty}^t e^{- u^2/2} d u / \sqrt{2 \pi}$ is the Gaussian c.d.f.. Note 

\begin{equation*}
 \mu' = \frac{1}{ \mu( \lambda K + \varepsilon B)} \mathbf{1}_{ \lambda K + \varepsilon B} \mu
\end{equation*}

the restriction of $\mu$ to the enlarged ball. As proved in \cite{boissard2011bounding}, Appendix 1,
 the Gaussian measure $\mu$ satisfies a $\bt_2(2 \sigma^2)$ inequality, 
hence a $\bt_1$ inequality with the same constant. We have

\begin{align*}
 W_1(\mu, \mu') & \leq \sqrt{ 2 \sigma^2 H( \mu' | \mu)} = \sqrt{-2 \sigma^2 \log \mu( \lambda K + \varepsilon B)} \\
{} & \leq \sqrt{-2 \sigma^2 \log \Phi( \lambda  + \Phi^{-1}(\mu(\varepsilon B)))}.
\end{align*}

On the other hand, 
denote $k = \mathcal{N}(\lambda K, \varepsilon)$ the covering number of $\lambda K$
(w.r.t. the norm of $E$).
Let $x_1, \ldots, x_k \in K$ be such that union of the balls
$B(x_i, \varepsilon)$ contains $\lambda K$.
From the triangle inequality we get the inclusion

\begin{equation*}
 \lambda K + \varepsilon B \subset \bigcup_{i = 1}^k B(x_i, 2 \varepsilon).
\end{equation*}

Choose a measurable map $T : \lambda K + \varepsilon B \rightarrow \{ x_1, \ldots, x_k \}$
such that for all $x$, $|x - T(x)| \leq 2 \varepsilon$.
The push-forward measure $\mu^k = T_\# \mu'$ has support in the finite set
$\{ x_1, \ldots, x_k \}$, and clearly

\begin{equation*}
 W_1( \mu', \mu^k) \leq 2 \varepsilon.
\end{equation*}

Choose $\varepsilon = t/16$, and

\begin{align}
 \lambda & = \Phi^{-1}( e^{-t^2/(128 \sigma^2)}) - \Phi^{-1}( \mu( \varepsilon B)) \\
{} \label{bound_lambda} & = \Upsilon^{-1}( e^{ - \psi( t/16)} ) + \Phi^{-1} (  e^{-t^2/(128 \sigma^2)})
\end{align}

where $\Upsilon (t) = \int_t^{+ \infty} e^{- u^2/2} d u / \sqrt{2 \pi}$
is the tail of the Gaussian distribution (we have used the fact that $\Phi^{-1} + \Upsilon^{-1} = 0$, which comes from symmetry of the Gaussian distribution).

Altogether, this ensures that $W_1(\mu, \mu^k) \leq t/4$.

\emph{Step 2. Bounding $\lambda$.}

We can use the elementary bound 
$\Upsilon(t) \leq e^{-t^2/2}$, $t \geq 0$ to get

\begin{equation*}
 \Upsilon^{-1}(u) \leq \sqrt{-2 \log u}, \quad 0 < u \leq 1/2
\end{equation*}

which yields $\Upsilon^{-1} (e^{ - \psi( t/16)}) \leq \sqrt{ \psi(t/16)}$
as soon as $\psi(t/16) \geq \log 2$.
Likewise, 

\begin{align*}
 \Phi^{-1}(e^{-t^2/128 \sigma^2}) & = \Upsilon^{-1}(1 - e^{-t^2/128 \sigma^2}) \\
{} & \leq \sqrt{2 \log \frac{1}{1 - e^{- t^2/128 \sigma^2}}} 
\end{align*}

as soon as $t^2/128 \sigma^2 \leq \log 2$. Moreover, for $u \leq \log 2$, we have
$1/(1 - e^{-u}) \leq 2 \log 2 / u$. Putting everything together, we get

\begin{equation} \label{bound_lambda_explicit}
 \lambda \leq \sqrt{ \psi(t/16)} + c\sqrt{\log \sigma/t}
\end{equation}

for some universal constant $c > 0$.
Observe that the first term in (\ref{bound_lambda_explicit}) will usually be much larger than the second one.

\emph{Step 3.}

From Theorem \ref{variant} we know that

\begin{equation*}
  \pr( W_2( \mu, L_n) \geq t) \leq K_t e^{-nt^2/16\sigma^2}
\end{equation*}

with

\begin{equation*}
 K_t = \exp \left[ \frac{1}{2 \sigma^2} \frac{k}{2} (\text{Diam } \{x_1, \ldots, x_k\})^2 \right].
\end{equation*}

The diameter is bounded by $\text{Diam }K = 2 \sigma \lambda \leq c \sigma ( \sqrt{ \psi(t/16)} + c\sqrt{\log \sigma/t})$.

We wish now to control $k = \mathcal{N}(\lambda K, t/16)$ in terms of the small ball function $\psi$. 
The two quantities are known to be connected : for example, Lemma 1 in \cite{kuelbs1993metric} gives the bound

\begin{equation*}
 \mathcal{N}( \lambda K, \varepsilon) \leq e^{\lambda^2/2 + \psi(\varepsilon/2)}.
\end{equation*}

Thus 

\begin{equation*}
 k \leq \exp \left[ \psi(t/16) + \psi(t/32) + c \log \sigma/t \right ].
\end{equation*}

With some elementary majorizations, this ends the proof.

\end{proof}

We can now sharpen the results of Proposition \ref{example_s_b_m}.
Let $\gamma$ denote the Wiener measure on $\mathcal{C}([0, 1], \R^d)$ endowed with the sup-norm, and denote by $\sigma^2$ its weak variance.
Let $\lambda_1$ be the first nonzero eigenvalue
of the Laplacian operator on the ball of $\R^d$ with homogeneous Dirichlet boundary conditions : 
it is well-known that the small ball function for the Brownian motion on $\R^d$ is equivalent to $\lambda_1/t^2$
when $t \rightarrow + \infty$. for $t$ small enough.

As a consequence, there exists $C = C(d)$ such that for small enough $t > 0$ we have

\begin{equation}
 W_1( L_n, \gamma) \leq \exp \exp \left[ C \lambda_1 / t^2  \right] e^{- n t^2/16 \sigma^2}.
\end{equation}

\subsection{Bounds in the dependent case : occupation measures of contractive Markov chains}

The results above can be extended to the convergence of the occupation measure for a
Markov process. As an example, we establish the following result.

\begin{theorem} \label{theorem_markov}
 Let $P(x, dy)$ be a Markov kernel on $\R^d$ such that
\begin{enumerate}
 \item the measures $P(x, .)$ satisfy a $\bt_1(C)$ inequality
 \item $W_1(P(x, .), P(y, .)) \leq r |x - y|$ for some $r < 1$.
\end{enumerate}

Let $\pi$ denote its invariant measure.
Let $(X_i)_{i \geq 0}$ denote the Markov chain associated with $P$ under $X_0 = 0$.

Set $a = \frac{2}{C} \left( \sqrt{4 m_1^2 + C \log 2} - 2 m_1 \right)$. There exists
$C_d > 0$ depending only on $d$ such that for $t \leq 2/a$,

\begin{equation*}
 \pr( W_1( L_n, \pi) \geq t) \leq K(n, t) \exp - n \frac{(1 - r)^2}{8C} t^2
\end{equation*}

where

\begin{equation*}
 K(n, t) = \exp \left[ \frac {m_1}{\sqrt{n C}} + C_d (\frac{1}{at} \log \frac{1}{at})^\frac{d}{2} \right]^2.
\end{equation*}

\end{theorem}

\begin{remark}
 The result is close to the one obtained in the independent case, and, as stressed in the introduction, it holds interest from the perspective of
numerical simulation, in cases where one cannot sample uniformly from a given probability distribution $\pi$ but may build a Markov chain
that admits $\pi$ as its invariant measure.
\end{remark}

\begin{remark}
 We comment on the assumptions on the transition kernel. The first one ensures that the $\bt_1$ inequality is propagated to the laws of $X_n$, $n \geq 1$.
As for the second one, it has appeared several times in the Markov chain literature
 (see e.g. \cite{djellout_guillin_wu}, \cite{ollivier2009ricci}, \cite{joulin2010curvature})
as a particular variant of the Dobrushin uniqueness condition for Gibbs measures. It has a nice geometric interpretation as a positive lower bound on 
the Ricci curvature of the Markov chain, put forward for example in \cite{ollivier2009ricci}. Heuristically,
this condition implies that the Markov chains started from two different points and suitably coupled tend to get closer.
\end{remark}

%% file: proof_of_main_thm.tex
\section{Proof of Theorem \ref{main_thm}} \label{proof_of_main_thm}

The starting point is the following result, obtained by Gozlan and Leonard (\cite{large_dev_gozlan_leonard}, see Chapter 6) 
by studying the tensorization properties of transportation inequalities.

\begin{lemma} \label{tensorization_lemma}
 
Suppose that $\mu \in \mathcal{P}(E)$ verifies a $\alpha(\mathcal{T}_d)$ inequality. Define on $E^n$ the metric

\begin{equation*}
d^{\oplus n}( (x_1, \ldots, x_n), (y_1, \ldots, y_n)) = \sum_{i = 1}^n d(x_i, y_i).
\end{equation*}

Then $\mu^{ \otimes n} \in \mathcal{P}(E^n)$ verifies a $\alpha'( \mathcal{T}_{d^{\oplus n}})$ inequality, where $\alpha'(t) = \frac{1}{n} \alpha( n t)$.
Hence, for all Lipschitz functionals
$Z : E^n \rightarrow \R$ (w.r.t. the distance $d^{\oplus n}$), we have the
concentration inequality

\begin{equation*}
 \mu^{\otimes n} ( Z \geq \int Z d \mu^{\otimes n} + t) \leq \exp - n  \alpha( \frac{t}{ n  \|Z \|_\text{Lip}}) \quad \text{for all } t \geq 0.
\end{equation*}

\end{lemma}

Let $X_i$ be an i.i.d. sample of $\mu$.
Recalling that 

\begin{equation*}
W_1(L_n, \mu) = \sup_{f 1-\text{Lip}} \frac{1}{n} \sum_{i = 1}^n f(X_i) - \int f d \mu
\end{equation*}

and that

\begin{equation*}
 (x_1, \ldots, x_n) \mapsto \sup _{f 1-\text{Lip}} \frac{1}{n} \sum_{i = 1}^n f(x_i) - \int f d \mu
\end{equation*}

is $\frac{1}{n}$-Lipschitz w.r.t. the distance $d^{\oplus n}$ on $E^n$ (as a supremum of $\frac{1}{n}$-Lipschitz functions),
the following ensues :

\begin{equation} \label{deviation_estimate_with_mean}
\pr( W_1(L_n, \mu) \geq \mathbb{E}[W_1(L_n, \mu)] + t) \leq \exp - n \alpha(t). 
\end{equation}

Therefore, we are led to seek a control on $\mathbb{E}[W_1(L_n, \mu)]$. This is what we do in the next lemma.

\begin{lemma} \label{lemma_mean}
 
Let $a > 0$ be such that 
$E_{a, 1} = \int e^{a d(x, x_0)} \mu(d x) \leq 2$.

Let $\delta > 0$ and $K \in E$ be a compact subset containing $x_0$.
Let $\mathcal{N}_\delta$ denote the covering number of order 
$\delta$ for the set $\mathcal{F}_K$ of $1$-Lipschitz functions on $K$
vanishing at $x_0$ (endowed with the uniform distance).

Also define $\sigma : [0, +\infty ) \rightarrow [1, + \infty)$ as the inverse function of $x \mapsto x \ln x - x + 1$ on $[1, + \infty)$. 

The following holds :

\begin{equation*}
\mathbb{E}[ W_1(L_n, \mu)] \leq 2 \delta + 8 \frac{1}{a} \frac{1}{\sigma (\frac{1}{\mu(K^c)})} +  \Gamma(\mathcal{N}_\delta, n)
\end{equation*}

where

\begin{equation*}
 \Gamma(  \mathcal{N}_\delta, n) = \inf_{\lambda > 0} \frac{1}{\lambda} [ \log \mathcal{N}_\delta + n \alpha^*(\frac{\lambda}{n}) ].
\end{equation*}

\end{lemma}

\begin{proof}
 
We denote by $\mathcal{F}$ 
the set of $1$-Lipschitz functions $f$ over $E$ such that $f(x_0) = 0$.
Let us denote

\begin{equation*}
 \Psi(f) = \int f d \mu - \int f d L_n,
\end{equation*}

we have for $f, g \in \mathcal{F}$ :

\begin{eqnarray*}
 | \Psi(f) - \Psi(g)| & \leq & \int |f - g| 1_K d \mu + \int |f - g| 1_K d L_n \\
{} & {} & + \int (|f| + |g|)1_{K^c} d \mu + \int (|f| + |g|) 1_{K^c} d L_n \\
{} & \leq & 2 \| f - g \|_{L^\infty(K)} + 2 \int d(x, x_0) 1_{K^c} d \mu + 2 \int d(x, x_0) 1_{K^c} d L_n
\end{eqnarray*}

When $f : E \rightarrow \mathbb{R}$ is a measurable function,
 denote by $f|_K$ its restriction to $K$. Notice that for every $g \in \mathcal{F}_K$, there
exists $f \in \mathcal{F}$ such that $f|_K = g$. Indeed, one may set

\begin{equation*}
 f(x) = 
\begin{cases}
 g(x) \text{ if } x \in K \\
\inf_{y \in K} f(y) + d(x, y) \text{ otherwise}
\end{cases}
\end{equation*}

and check that $f$ is $1$-Lipschitz over $E$.

By definition of $\mathcal{N}_\delta$, there exist
functions $g_1, \ldots, g_{\mathcal{N}_\delta} \in \mathcal{F}_K$ 
such that the balls of center $g_i$ and radius $\delta$ 
(for the uniform distance) cover $\mathcal{F}_K$. We can extend 
these functions to functions $f_i \in \mathcal{F}$ as noted above.

Consider $f \in \mathcal{F}$ and choose $f_i$ such that $|f - f_i| \leq \delta$ on $K$ :

\begin{eqnarray*}
 \Psi(f) & \leq &   | \Psi(f) - \Psi(f_i)| + \Psi(f_i) \\
{} & \leq & \Psi(f_i) + 2 \delta + 2 \int d(x, x_0) 1_{K^c} d \mu + 2 \int d(x, x_0) 1_{K^c} d L_n \\
{} & \leq & \max_{j = 1, \ldots, \mathcal{N}_\delta} \Psi(f_j) + 2 \delta + 2 \int d(x, x_0) 1_{K^c} d \mu + 2 \int d(x, x_0) 1_{K^c} d L_n
\end{eqnarray*}

The right-hand side in the last line does not depend on $f$, so it is also greater than $W_1( L_n, \mu) = \sup_\mathcal{F} \Psi(f)$.

We pass to expectations, and bound the terms on the right.
We use Orlicz-H\"{o}lder's inequality with the pair of conjugate Young functions

\begin{align*}
  \tau(x) & =  \begin{cases}
                 0 \text{ if } x \leq 1 \\
		 x \log x - x + 1 \text{ otherwise}
                \end{cases} \\
  \tau^*(x) & = e^x - 1
\end{align*}

(for definitions and a proof of Orlicz-H\"{o}lder's inequality, the reader may refer to \cite{rao1991theory}, Chapter 10).
We get

\begin{equation*}
 \int d(x, x_0) 1_{K^c} d \mu \leq 2 \| 1_{K^c} \|_\tau \| d(x, x_0) \|_{\tau^*}
\end{equation*}

where

\begin{equation*}
 \| 1_{K^c} \|_\tau = \inf \{ \theta > 0, \quad \int \tau \left( \frac{1_{K^c}}{\theta} \right) d \mu \leq 1 \}
\end{equation*}

and

\begin{equation*}
 \| d(x, x_0) \|_{\tau^*} = \inf \{ \theta > 0, \quad \int \left[ e^{\frac{d(x, x_0)}{\theta}} - 1 \right] d \mu \leq 1 \}.
\end{equation*}

It is easily seen that $ \| 1_{K^c} \|_\tau = 1/\sigma (1/\mu(K^c))$.
And we assumed that $a$ is such that $E_{a, 1} = \int \exp a d(x, x_0) d \mu \leq 2$, so
$\| d(x, x_0) \|_{\tau^*} \leq 1/a$.
Altogether, this yields

\begin{equation*}
 \int d(x, x_0) 1_{K^c} d \mu \leq 2 \frac{1}{a}\frac{1}{\sigma (\frac{1}{\mu(K^c)})}.
\end{equation*}

Also, if $X_1, \ldots, X_n$ are i.i.d. variables of law $\mu$,

\begin{equation*}
  \mathbb{E} [ \int d(x, x_0) 1_{K^c} d L_n ] = \mathbb{E} [ d(X_1, x_0) 1_{K^c}(X_1) ] \leq \frac{2}{a} \frac{1}{\sigma (1/\mu(K^c))}
\end{equation*}

as seen above. Putting this together yields the inequality

\begin{equation*}
 \mathbb{E} [W_1( L_n, \mu)] \leq 2 \delta + \frac{8}{a} \frac{1}{\sigma (1/\mu(K^c))} + \mathbb{E} [ \max_{j = 1, \ldots, \mathcal{N}_\delta} \Psi(f_j) ].
\end{equation*}

The remaining term can be bounded by a form of maximal inequality.
 First fix some $i$ and $\lambda > 0$ : we have

\begin{eqnarray*}
 \mathbb{E}[ \exp \lambda \Psi(f_i) ] & = & \mathbb{E} [ \exp \frac{\lambda}{n} \sum_{j = 1}^n (f(X_j) - \int f d \mu) ] \\
{} & = & ( \mathbb{E} [ \exp \frac{\lambda}{n} (f(X_1) - \int f d \mu) ] )^n \\
{} & \leq & e^ {n \alpha^\circledast(\lambda/n)}.
\end{eqnarray*}

In the last line, we have used estimate (\ref{cond_laplace}).
Using Jensen's inequality, we may then write

\begin{eqnarray*}
 \mathbb{E} [ \max_{j = 1, \ldots, \mathcal{N}_\delta} \Psi(f_j) ] & \leq & \frac{1}{\lambda} \log \mathbb{E} [ \max_{j = 1, \ldots, \mathcal{N}_\delta} \exp \lambda \Psi(f_j)  ] \\
{} & \leq & \frac{1}{\lambda} \log \sum_{j = 1}^{\mathcal{N}_\delta} \mathbb{E} [ \exp \lambda \Psi(f_j)] \\
{} & \leq & \frac{1}{\lambda} [ \log \mathcal{N}_\delta + n \alpha^*(\frac{\lambda}{n}) ]
\end{eqnarray*}

So minimizing in $\lambda$ we have 

\begin{equation*}
 \mathbb{E} [ \max_{j = 1, \ldots, \mathcal{N}_\delta} \Psi(f_j) ] \leq \Gamma(  \mathcal{N}_\delta, n).
\end{equation*}

Bringing it all together finishes the proof of the lemma.

\end{proof}

We can now finish the proof of Theorem \ref{main_thm}.

\begin{proof}
 
Come back to the deviation bound (\ref{deviation_estimate_with_mean}).
Choose $\delta = t/8$, and choose $K$ such that 

\begin{equation*}
\mu(K^c) \leq \left[ \frac{32}{a t} \log \frac{32}{a t} - \frac{32}{a t} + 1 \right]^{-1}.
\end{equation*}

 We thus have
$2 \delta + 8 [a \sigma (1/\mu(K^c)) ]^{-1} \leq t/2$, which implies

\begin{equation} \label{general_estimate_mean}
\mathbb{E}( W_1(L_n,\mu) \leq t/2 + \Gamma(\mathcal{C}_t, n) 
\end{equation}

 and so

\begin{equation*}
 \pr( W_1(L_n, \mu) \geq  t) \leq \exp -n \alpha(\frac{t}{2} -\Gamma(\mathcal{N}_\delta, n)),
\end{equation*}

with the convention $\alpha(y) = 0$ if $y < 0$.

\end{proof}

%% file: variant.tex
\section{Proof of Theorem \ref{variant}} \label{section_variant}

In this section, we provide a different approach to our result in the independent case.
As earlier we first aim to get a bound
 on the speed of convergence on the average $W_1$ distance between empirical and true measure.
The lemma below provides another way to obtain such an estimate.

\begin{lemma} \label{lemma_alternative}

   Let $\mu^k \in \Pro(E)$ be a finitely supported measure such that $|\text{Supp }\mu^k| \leq k$.
Let $D( \mu^k) = \text{Diam Supp } \mu^k$ be the diameter of $\text{Supp } \mu^k$. 
The following holds :

\begin{equation*}
 \mathbb{E} W_1(\mu, L_n) \leq 2 W_1 (\mu, \mu^k) + D(\mu^k) \sqrt{k / n}.
\end{equation*}

\end{lemma}

\begin{proof}
 
 Let $\pi_\text{opt}$ be an optimal
coupling of $\mu$ and $\mu^k$ (it exists : see e.g. Theorem 4.1 in \cite{optimal_transport_villani}), and let
$(X_i, Y_i)$, $1 \leq i \leq n$, be i.i.d. variables on $E \times E$ with common law $\pi_\text{opt}$.

Let $L_n = 1/n \sum_{i = 1}^n \delta_{X_i}$ and $L_n^k = 1/n \sum_{i = 1}^n \delta_{Y_i}$.
By the triangle inequality, we have

\begin{equation*}
 W_1(L_n, \mu) \leq W_1(L_n, L_n^k) + W_1(\mu, \mu^k) + W_1(\mu^k, L_n^k).
\end{equation*}

With our choice of coupling for $L_n$ and $L_n^k$ it is easily seen that

\begin{equation*}
 \mathbb{E} W_1 (L_n, L_n^k) \leq W_1 (\mu, \mu^k)
\end{equation*}

Let us take care of the last term. We use Lemma \ref{lemma_measures_finite_support} below to obtain that

\begin{align*}
 \mathbb{E} W_1 (L_n^k, \mu^k) & \leq D(\mu^k) \mathbb{E} \left( 1 - \sum_{i = 1}^k \mu^k(x_i) \wedge L_n^k(x_i) \right)\\
{} & = D(\mu^k) \sum_{i = 1}^k \mathbb{E} ( \mu^k(x_i) - \mu^k(x_i) \wedge L_n^k(x_i) ) \\
{} & \leq D(\mu^k) \sum_{i = 1}^k \mathbb{E} | \mu^k(x_i) - L_n^k(x_i)| \\
{} & \leq \frac{D(\mu^k)}{n} \sum_{i = 1}^k \sqrt{ \mathbb{E} |  n\mu^k(x_i) - n L_n^k(x_i)|^2 }.
\end{align*}
 
Observe that the variables $n L_n^k(x_i)$ follow binomial laws with parameter $\mu^k(x_i)$ and $n$. We get :

\begin{equation*}
 \mathbb{E} W_1(\mu^k, L_n^k) \leq \frac{D(\mu^k)}{n}  \sum_{i = 1}^k  \sqrt{n \mu^k(x_i) (1 - \mu^k(x_i))} \leq D(\mu^k) \sqrt{k/n}
\end{equation*}

(the last inequality being a consequence of the Cauchy-Schwarz inequality).

\end{proof}

\begin{lemma} \label{lemma_measures_finite_support} 
Let $\mu, \nu$ be probability measures with support in a finite metric space $\{x_1, \ldots, x_k\}$
of diameter bounded by $D$. Then 

\begin{equation*}
 W_1(\mu, \nu) \leq D \left(1 - \sum_{i = 1}^k \left( \mu(x_i) \wedge \nu(x_i) \right)\right).
\end{equation*}
 
\end{lemma}

\begin{proof}
 We build a coupling of $\mu$ and $\nu$ that leaves as much mass in place as possible, 
in the following fashion : 
set $f(x_i) = \mu(x_i) \wedge \nu(x_i)$ and $\lambda = \sum_{i _ 1}^k f_i$.
Set $q(x_i) = f_i/ \lambda$, and define the measures

\begin{align*}
 \mu_1 & = \frac{1}{1- \lambda} (\mu - f) \\
\nu_1 & = \frac{1}{1- \lambda} (\nu - f).
\end{align*}

Finally, build independent random variables
 $X_1 \sim \mu_1$, $Y_1 \sim \nu_1$, $Z \sim q$ and $B$ with Bernoulli law of parameter $\lambda$.
Define

\begin{equation*}
 X = (1-B)X_1 + BZ, \quad Y = (1-B)Y_1 + BZ.
\end{equation*}

It is an easy check that $X \sim \mu$, $Y \sim \nu$.

Thus we have the bound

\begin{align*}
 W_1 (\mu, \nu) & \leq \mathbb{E} |X - Y| \\
{} & = (1- \lambda) \mathbb{E}|X_1 - Y_1| \leq D (1 - \lambda)
\end{align*}

and this concludes the proof.

\end{proof}

\begin{proof}[Proof of Theorem \ref{variant}]

As stated earlier, we have the concentration bound

\begin{equation*}
 \pr( W_p(L_n, \mu) \geq t + \mathbb{E} W_p(L_n, \mu) ) \leq e^{-nt^2 / C}.
\end{equation*}

The proof is concluded by arguments similar to the ones used before, calling upon Lemma \ref{lemma_alternative} to bound the mean.

\end{proof}

%% file: dependent_case.tex
\section{Proofs in the dependent case} \label{dependent_case}

Before proving Theorem \ref{theorem_markov}, we establish a more general result in the spirit of Lemma \ref{lemma_mean}.

As earlier, the first ingredient we need to apply our strategy of proof is
 a tensorization property for the transport-entropy inequalities 
in the case of non-independent sequences. To this end, we restate results from 
\cite{djellout_guillin_wu}, where only $\bt_1$
inequalities were investigated, in our framework.

For $x = (x_1, \ldots, x_n) \in E^n$, and $1 \leq i \leq n$, denote $x^i = (x_1, \ldots, x_i)$.
Endow $E^n$ with the distance $d_1(x, y) = \sum_{i _ 1}^n d(x_i, y_i)$.
Let $\nu \in \Pro(E^n)$, the notation $\nu^i(d x_1, \ldots, d x_i)$ stands for the marginal measure
on $E^i$, and $\nu^i(. | x^{i-1})$ stands for the regular 
conditional law of $x_i$ knowing $x^{i-1}$, or in other words the conditional
disintegration of $\nu^i$ with respect to $\nu^{i-1}$ at $x^{I-1}$(its existence is assumed throughout).

The next theorem is a slight extension of Theorem 2.11 in \cite{djellout_guillin_wu}.
 Its proof can be adapted without difficulty, and we omit it here.

\begin{theorem} \label{theorem_concentration_dependent_sequences}

 Let $\nu \in \Pro(E^n)$ be a probability measure such that

\begin{enumerate}
 \item For all $i \geq 1$ and all  $x^{i-1} \in E^{i-1}$ ($E^0 = \{x_0\})$, $\nu^i(. |x^{i-1})$
satisfies a $\alpha(\mathcal{T}_d)$ inequality, and
 \item 
There exists $S > 0$ such that for every $1$-Lipschitz function

\begin{equation*}
 f : (x_{k+1}, \ldots, x_n) \rightarrow f(x_{k+1}, \ldots, x_n),
\end{equation*}

for all $x^{k-1} \in E^{k-1}$ and $x_k, y_k \in E$, we have

\begin{equation} \label{contraction_property}
\begin{split}
 | \mathbb{E}_\nu\left( f(X_{k+1}, \ldots, X_n) | X^k = (x^{k-1}, x_k) \right) - \\ 
\mathbb{E}_\nu \left( f(X_{k+1}, \ldots, X_n) | X^k = (x^{k-1}, y_k) \right) | \\
\leq S d(x_k, y_k)
\end{split}
\end{equation}
\end{enumerate}

Then $\nu$ verifies the transportation inequality $\tilde{\alpha}( \mathcal{T}_d) \leq H$
 with 

\begin{equation*}
 \tilde{\alpha}(t) = n \alpha(\frac{1}{n (1+S)}  t).
\end{equation*}
\end{theorem}

In the case of a homogeneous Markov chain $(X_n)_{n \in \N}$ with transition kernel $P(x, dy)$, 
the next
 proposition gives sufficient conditions 
on the transition probabilities for
the laws of the variables $X_n$ and the path-level law
of $(X_1, \ldots, X_n)$ to satisfy some transportation inequalities. 
Once again the statement and its proof are
 adaptations of the corresponding Proposition 2.10 of \cite{djellout_guillin_wu}.

\begin{prop} \label{marginals_markov}
 Let $P(x, dy)$ be a Markov kernel such that

\begin{enumerate}
 \item \label{cond_a}the transition measures $P(x, .)$ satisfies $\alpha(\mathcal{T}_d) \leq H$ for all $x \in E$, and
 \item \label{cond_b} $W_1(P(x, .), P(y, .) ) \leq r d(x, y)$ for all $x, y \in E$ and some $r < 1$. 
\end{enumerate}

Then there exists a unique invariant probability measure $\pi$ for the Markov chain associated to $P$, and the measures
$P^n(x, .)$ and $\pi$ satisfy $\alpha'(\mathcal{T}_d) \leq H$, where $\alpha'(t) = \frac{1}{1 - r} \alpha( (1 - r) t)$.

Moreover, under these hypotheses, the conditions of Theorem \ref{theorem_concentration_dependent_sequences} are verified
with $S = \frac{r}{1-r}$ so that the law $P_n$ of the $n$-uple $(X_1, \ldots, X_n)$ under $X_0 = x_0 \in E$ 
verifies $\tilde{\alpha}(\mathcal{T}_d) \leq H$ where $\tilde{\alpha}(t) = n \alpha(\frac{1 - r}{n}  t)$.

\end{prop}

\begin{proof}
 The first claim is obtained exactly as in the proof of Proposition 2.10 in \cite{djellout_guillin_wu}, observing that
the contraction condition \ref{cond_b} is equivalent to

\begin{equation*}
 W_1(\nu_1 P, \nu_2 P) \leq r W_1( \nu_1, \nu_2) \quad \text{for all } \nu_1, \nu_2 \in \Pro_1(E)
\end{equation*}

and also to

\begin{equation*}
 \| Pf \|_\text{Lip} \leq r \|f\|_\text{Lip} \quad \text{for all } f.
\end{equation*}

This entails that whenever $f$ is $1$-Lipschitz, $P^n f$ is $r^n$-Lipschitz. Now, by condition \ref{cond_a}, we have

\begin{equation*}
\begin{aligned}
P^n(e^{s f}) & \leq P^{n-1}\left( \exp \left( sPf + \alpha^\circledast(s) \right) \right)   \\
{} & \leq P^{n-2} \left(  \left(  s P^2 f + \alpha^\circledast(s) + \alpha^\circledast(rs) \right) \right) \\
{} & \leq \ldots \\
{} & \leq \exp \left( \left( s P^n f + \alpha^\circledast (s) + \ldots + \alpha^\circledast (r^n s) \right) \right).
\end{aligned}
\end{equation*}

As $\alpha^\circledast$ is convex and vanishes at $0$, we have $\alpha^\circledast (r t) \leq r \alpha^\circledast (t)$ for
all $t \geq 0$. Thus,

\begin{equation*}
 P^n(e^{s f} \leq \exp \left( s P^n f + \sum_{k = 0}^{ + \infty} r^k \alpha^\circledast (s) \right) = \exp \left( s P^n f + \frac{1}{1 - r} \alpha^\circledast (s) \right)
\end{equation*}

It remains only to check that $\frac{1}{1 - r} \alpha^\circledast$ is the monotone conjugate of $\alpha'$ and to invoke Proposition \ref{bobkov_gotze_characterization}.

Moving on to the final claim, since the process is homogeneous, to ensure that (\ref{contraction_property}) is satisfied,
 we need only show that for all $k \geq 1$, for
all $f : E^k \rightarrow \R$ $1$-Lipschitz w.r.t. $d_1$, the function

\begin{equation*}
x \mapsto \mathbb{E} \left[f(X_1, \ldots, X_k)|X_0 = x \right]
\end{equation*}

is $\frac{r}{1 - r}$-Lipschitz. 
We show the following : if $g : E^2 \rightarrow \R$ is 
such that for all $x_1, x_2 \in E$ the functions
$g(., x_2)$, resp. $g(x_1, .)$, are $1$-Lipschitz, resp. $\lambda$-Lipschitz,
 then the function

\begin{equation*}
 x_1 \mapsto \int g(x_1, x_2) P(x_1, d x_2)
\end{equation*}

is $(1 + \lambda r)$-Lipschitz.
Indeed, 

\begin{align*}
 | \int g(x_1, x_2) & P(x_1, d x_2) - \int g(y_1, x_2) P(y_1, d x_2) | \\
{} & \leq  \int |g(x_1, x_2)  - g(y_1, x_2) |P(x_1, d x_2) \\
{} & \quad + |\int g(y_1, x_2) (P(x_1, d x_2) -  P(y_1, d x_2)) | \\
{} & \leq (1 + \lambda r) d(x_1, y_1).
\end{align*}

It follows easily by induction that the function

\begin{equation*}
f_k : x_1 \mapsto \int f(x_1, \ldots, x_k) P(x_{k-1}, d x_k) \ldots P(x_1, d x_2)
\end{equation*}

has Lipschitz norm bounded by $1 + r + \ldots r^k \leq \frac{1}{1 - r}$.
Hence the function $x \mapsto \int  f_k(x_1) P(x, d x_1)$ has Lipschitz norm bounded by $\frac{r}{1 - r}$.
But this function is precisely

\begin{equation*}
x \mapsto \mathbb{E} \left[f(X_1, \ldots, X_k)|X_0 = x \right]
\end{equation*}

and the proof is complete.
\end{proof}

We are in position to prove the analogue of Lemma \ref{lemma_mean} in the Markov case.

\begin{lemma}
Consider the Markov chain associated to a transition kernel $P$ as in Proposition \ref{marginals_markov}.
Let $P_n$ denote the law of the Markov path $(X_1, \ldots, X_n)$ associated to $P$ under $X_0 = x_0$.
Introduce the averaged occupation measure $\pi_n = \mathbb{E}_{P_n} (L_n)$ and the invariant measure $\pi$.
Let $m_1 = \int d(x, x_0) \pi(dx)$.

Suppose that there exists $a > 0$ such that for all $i \geq 1$
$E_{a, i} = \int e^{a d(x, x_0)} P^i(d x) \leq 2$.

Let $\delta > 0$ and $K \in E$ be a compact subset containing $x_0$.
Let $\mathcal{N}_\delta$ denote the metric entropy of order 
$\delta$ for the set $\mathcal{F}_K$ of $1$-Lipschitz functions on $K$
vanishing at $x_0$ (endowed with the uniform distance).
Also define $\sigma : [0, +\infty ) \rightarrow [1, + \infty)$ as the inverse function of $x \mapsto x \ln x - x + 1$ on $[1, + \infty)$. 

The following holds :

\begin{align*}
\mathbb{E}_{P_n}[ W_1(L_n, \pi_n)]  & \leq 2 \delta +  \frac{8}{a} \frac{1}{n}\sum_{i = 1}^{n}\frac{1}{\sigma (\frac{1}{P^i(K^c)})} +  \Gamma(\mathcal{N}_\delta, n) \\
W_1( \pi_n, \pi) & \leq \frac{m_1}{n(1 - r)}.
\end{align*}
where $\Gamma(\mathcal{N}_\delta, n) = \inf_{\lambda > 0} \frac{1}{\lambda} \left[ \log \mathcal{N}_\delta + n \alpha^\circledast (\frac{\lambda}{n(1 - r)}) \right] $ 

\end{lemma}

\begin{proof}
 
Convergence to the equilibrium measure is dealt with using the contraction hypothesis. Indeed, by convexity of the map
$\mu \mapsto W_1(\mu, \pi)$, we first have 

\begin{equation*}
W_1(\pi_n, \pi) \leq \frac{1}{n} \sum_{i = 1}^n W_1( P^i(x_0, .), \pi).
\end{equation*}

Now, using that the contraction property (\ref{cond_b}) in Proposition \ref{marginals_markov} is equivalent to the inequality
$W_1( \mu_1 P, \mu_2 P) \leq r W_1( \mu_1, \mu_2)$ for all $\mu_1, \mu_2 \in \Pro_1(E)$, and using the fact that $\pi$
is $P$-invariant,

\begin{equation*}
W_1(\pi_n, \pi) \leq \frac{1}{n} \sum_{i = 1}^n r^i W_1( \delta_{x_0}, \pi) \leq \frac{W_1( \delta_{x_0}, \pi)}{n(1 - r)} = \frac{m_1}{n(1 - r)}.
\end{equation*}

In order to take care of the second term, we will use the same strategy (and notations) as in the independent case. Introduce once again a compact subset
$K \subset E$ and a covering of $\mathcal{F}_K$ by functions $f_1, \ldots, f_{\mathcal{N}_\delta}$ suitably extendend to $E$.
With the same arguments as before, we get

\begin{align*}
\mathbb{E}_{P_n}W_1( L_n, \pi_n) \leq &  \mathbb{E}_{P_n} (\max_{j = 1, \ldots, \mathcal{N}_\delta} \Psi(f_j)) + 2 \delta + 2 \int d(x, x_0) 1_{K^c} d \pi_n  \\
{} &  + 2 \mathbb{E}_{P_n} (\int d(x, x_0) 1_{K^c} d L_n)
\end{align*}

Then,

\begin{equation*}
 \int d(x_0, y) \pi_n(d y) = \frac{1}{n} \sum_{i = 1}^n \int d(x_0, y)1_{K^c}  P^i(x_0, d y).
\end{equation*}

As before we can use Orlicz-H\"{o}lder's inequality to recover the bound

\begin{equation*}
\int d(x_0, y) d \pi_n \leq \frac{2}{a} \frac{1}{n} \sum_{i = 1}^{n} \frac{1}{\sigma (\frac{1}{P^i(K^c)})}.
\end{equation*}

And likewise,

\begin{align*}
 \mathbb{E} (\int d(x, x_0) 1_{K^c} d L_n) & = \mathbb{E} \displaystyle \left[ \frac{1}{n} \sum_{i = 1}^n d(x_0, X_i) 1_{K^c} \right] \\
{} & = \frac{1}{n} \sum_{i = 1}^n \int d(x_0, y) 1_{K^c} P^i(x_0, dy)
\end{align*}

and we have the same bound as above.

As for the last term remaining : it will be possible to use the maximal inequality techniques just as in the proof of Theorem \ref{main_thm}, 
provided that we can find bounds on the terms $\mathbb{E} \displaystyle \left[ \exp \lambda \Psi (f_j) \right]$, where this time

\begin{equation*}
 \Psi(f) = \int f d L_n - \int f d \pi_n.
\end{equation*}

Denote 
                                                                         
\begin{equation*}
 F_j(x_1, \ldots, x_n) = \frac{1}{n} \sum_{i = 1}^n f_j(x_i).
\end{equation*}
         
This is a $\frac{1}{n}$-Lipschitz function on $E^n$. Since $P_n$ satisfies a $\tilde{\alpha}(\mathcal{T}_d) \leq H$ inequality, 
we have

\begin{equation*}
 \int \exp \lambda F_j d P_n \leq \exp \left[  \lambda \int F_j d P_n +  n \alpha^\circledast (\frac{\lambda}{n(1 - r)}) \right].
\end{equation*}

But this is exactly the bound

\begin{equation*}
 \mathbb{E} \displaystyle \left[ \exp \lambda \Psi(f_j) \right] \leq e^{n \alpha^\circledast (\frac{\lambda}{n(1 - r)})}.
\end{equation*}

We may then proceed as in the independent case and obtain

\begin{equation*}
 \mathbb{E} [ \max_{j = 1, \ldots, \mathcal{N}_\delta} \Psi(f_j) ] \leq \inf_{\lambda > 0} \frac{1}{\lambda} \left[ \log \mathcal{N}_\delta + n \alpha^\circledast (\frac{\lambda}{n(1 - r)}) \right].
\end{equation*}

\end{proof}

For any Lipschitz function $f : E^n \rightarrow \R$ (w.r.t. $d_1$), we have the concentration inequality

\begin{equation*}
 P_n (x \in E^n, \quad f(x) \geq \int f d P_n + t) \leq \exp - n \alpha \left( \frac{ (1-r)t}{ n \|f\|_\text{Lip}} \right).
\end{equation*}

Remembering that $E^n \ni x \mapsto W_1( L_n^x, \pi_n)$ is $\frac{1}{n}$-Lipschitz, we get the bound

\begin{equation} \label{concentration_markov}
 \pr( W_1(L_n, \pi_n) \geq \mathbb{E}_{P_n}[ W_1(L_n, \pi_n) ] + t ) \leq \exp - n \alpha \left( (1 - r)t\right).
\end{equation}

Thanks to the triangular inequality $W_1(L_n, \pi_n) \geq W_1(L_n, \pi) - W_1(\pi_n, \pi)$, it holds that

\begin{equation} \label{concentration_ergodic}
 \pr ( W_1( L_n, \pi) \geq W_1(\pi_n, \pi) + \mathbb{E}_{P_n} [W_1(L_n, \pi_n)] + t) \leq \exp - n \alpha \left((1 - r)t \right).
\end{equation}

This in turn leads to an estimate on the deviations, under
the condition that we may exhibit a compact set with large measure for all the measures $P^i$.
We now move on to the proof of Theorem \ref{theorem_markov}.

\begin{proof}[Proof of Theorem \ref{theorem_markov}]
 Fix $\delta = t/8$.
Set $m_1^i = \int |x| P^i(d x)$. We have

\begin{align*}
 m_1^i & \leq m_1 + W_1( P^i, \pi) \\
{} & \leq m_1 + r^i W_1( \delta_0, \pi) \\
{} & \leq 2 m_1.
\end{align*}

Thus

\begin{equation*}
 \int e^{a |x|} P^i(dx) \leq e^{a m_1^i + C a^2/4} \leq e^{2 m_1 a + C a^2 / 4}.
\end{equation*}

With $a$ as in the theorem, the above ensures that
$\int e^{a |x|} P^i(dx) \leq 2$.

Let $B_R$ denote the ball of center $0$ and radius $R$ :
we have $P^i(B_R^c) \leq 2 e^{- a R}$. Let

\begin{equation*}
 R = \frac{1}{a} \log 2 \sigma^{-1}(\frac{32}{ a t}).
\end{equation*}

so that $ 2 \delta +  \frac{8}{a} \frac{1}{n}\sum_{i = 1}^{n}\frac{1}{\sigma (\frac{1}{P^i(K^c)})} \leq t/2$.

As $\alpha(t) = \frac{1}{C} t^2$ we can compute 

\begin{equation*}
\Gamma(\mathcal{N}_\delta, n) = \frac{1}{1-r}\sqrt{\frac{C}{n}} \sqrt{\log \mathcal{N}_\delta}.
\end{equation*}

We have chosen $K = B_R$ and $\delta = t/8$. 
Working as in the proof of Proposition \ref{result_R}, when $t \leq 2/a$, we can bound $\log \mathcal{N}_\delta$ by

\begin{equation*}
 \log \mathcal{N}_\delta \leq C_d (\frac{1}{at} \log \frac{1}{at})^d
\end{equation*}

where $C_d$ is a numerical constant depending on the dimension $d$.
Plugging the above estimates in (\ref{concentration_ergodic}) and using the inequality $(u-v)^2 \geq u^2/2 - v^2$
gives the desired result.

\end{proof}

%% file: appendix_transportation_inequalities.tex
\section{Some facts on transportation inequalities}  \label{appendix_transport_inequalities}

A crucial feature of transportation inequalities is that they imply the concentration of measure phenomenon, a fact first discovered by Marton
(\cite{marton1996bounding}). The following proposition is obtained by a straightforward adaptation of her famous argument :

\begin{prop}
 If $\mu$ verifies a $\alpha(\mathcal{T}_d)$ inequality, then for all measurable sets $A \subset E$ with $\mu(A) \geq \frac{1}{2}$ and $r \geq r_0 = \alpha^{-1}(\log 2)$,

\begin{equation*}
 \mu(A^r) \geq 1 - e^{-\alpha(r - r_0)}
\end{equation*}

where $A^r = \{ x \in E, \; d(x, A) \leq r \}$.

Moreover, let $X$ be a r.v. with law $\mu$. For all $1$-Lipschitz functions $f : E \rightarrow \R$ and all $r \geq r_0$, we have

\begin{equation*}
 \pr ( f(X) \geq m_f + r) \leq e^{- \alpha(r - r_0)}
\end{equation*}

where $m_f$ denotes a median of $f$.

\end{prop}

Bobkov and G\"{o}tze (\cite{bobkov1999exponential}) were the first to obtain an equivalent dual formulation of transportation inequalities.
We present it here in a more general form obtained by Gozlan and Leonard (see \cite{gozlan2010transport}), in the case when the transportation cost function is the distance.

\begin{definition}
 
Let $\alpha : [0, + \infty) \rightarrow \R$ be convex, increasing, left-continuous and vanishing at $0$.
The monotone conjugate of $\alpha$ is

\begin{equation*}
\alpha^\circledast (s) = \sup_{t \geq 0} st - \alpha(t). 
\end{equation*}

\end{definition}

\begin{prop}[\cite{gozlan2010transport}] \label{bobkov_gotze_characterization}
 
Assume that $d$ is a metric defining the topology of $E$, and that there exist $a > 0$, $x_0 \in E$ such that
 $\int \exp[ a d(x, x_0) ]\mu (d x) < + \infty$.

Then $\mu$ satisfies the $\alpha(\mathcal{T}_d)$ inequality

\begin{equation*}
 \alpha( \mathcal{T}_d (\mu, \nu) ) \leq H( \nu | \mu)
\end{equation*}

for all $\nu \in \Pro(E)$ with finite first moment if and only if 
for all $f : E \rightarrow \R$ $1$-Lipschitz and all $\lambda > 0$,

\begin{equation} \label{cond_laplace}
 \int e^{\lambda (f(x) - \int f d \mu)} \mu(d x) \leq e^{ \alpha^\circledast (\lambda)}.
\end{equation}

\end{prop}

In the case $\bt_1(C)$, Condition (\ref{cond_laplace}) becomes : for all $1$-Lipschitz $f : E \rightarrow \R$ and $\lambda > 0$,

\begin{equation}
 \int e^{\lambda(f - \int f d \mu)} \mu (x) \leq e^{C\lambda^2/4}.
\end{equation}

\subsubsection{Integral criteria} \label{integral_criteria}

An interesting feature of transportation inequalities is that some of them are 
characterized by simple moment conditions, making it tractable to obtain their existence. 
In \cite{djellout_guillin_wu},
Djellout, Guillin and Wu showed that $\mu$ satisfies a $\bt_1$ inequality if and only if 
$\int \exp [ a_0 d^2(x_0, y)] \mu(dy) < + \infty$
for some $a_0$ and some $x_0$. They also connect the value of $a_0$ and of the Gaussian moment with the value of
the constant $C$ appearing in the transportation inequality. More generally,
Gozlan and Leonard provide in \cite{large_dev_gozlan_leonard} a nice criterion
 to ensure that 
a $\alpha( \bt_d)$ inequality holds. We only quote here one side of what is actually an equivalence :

\begin{theorem}
 Let $\mu \in \Pro(E)$. 
Suppose there exists $a > 0$ with $\int e^{a d(x_0, x)} \mu(d x) \leq 2$ for some $x_0 \in E$, and
a convex, increasing function $\gamma$ on $[0, + \infty)$ vanishing at $0$ and $x_1 \in E$ such that

\begin{equation*}
 \int \exp \gamma (d(x_1, x) ) \mu(d x) \leq B < + \infty.
\end{equation*}

Then $\mu$ satisfies the $\alpha(\mathcal{T}_d)$ inequality

\begin{equation*}
 \alpha( W_1( \mu, \nu) ) \leq H( \nu | \mu)
\end{equation*}

for all $\nu \in \Pro(E)$ with finite first moment, with 

\begin{equation*}
\alpha(x) = \max \left( (\sqrt{ax + 1} - 1)^2, \; 2 \gamma( \frac{x}{2} - 2 \log B )\right).  
\end{equation*}

\end{theorem}

One particular instance of the result above was first obtained by Bolley and Villani, with sharper constants, in the case when $\mu$
only has a finite exponential moment (\cite{bolley2005weighted}), Corollary 2.6). Their technique involves the study of weighted Pinsker inequalities, and
encompasses more generally costs of the form $d^p$, $p \geq 1$ (we give only the case $p = 1$ here).

\begin{theorem}
 Let $a > 0$ be such that $E_{a, 1} = \int e^{a d(x_0,x)} \mu(d x) < + \infty$. Then for $\nu \in \Pro_1(E)$, we have

\begin{equation*}
 W_1(\mu, \nu) \leq C \displaystyle \left( H( \nu | \mu) + \sqrt{ H(\nu | \mu)  } \right)
\end{equation*}

where $C = \frac{2}{a} \left( \frac{3}{2} + \log E_{a, 1} \right) < + \infty.$
\end{theorem}

And in the case when $\mu$ admits a finite Gaussian moment, the following holds (\cite{bolley2005weighted}, Corollary 2.4) :

\begin{theorem}
 Let $a > 0$ be such that $E_{a, 2} = \int e^{a d^2(x_0,x)} \mu(d x) < + \infty$. Then $\mu$ satisfies a $\bt_1(C)$ inequality
where $C = \frac{2}{a} \left( 1 + \log E_{a, 2} \right) < + \infty.$
\end{theorem}

%% file: appendix_1.tex
\section{Covering numbers of the set of $1$-Lipschitz functions}

In this section, we provide bounds for the covering numbers of the set of $1$-Lipschitz functions over a precompact space.

Note that these results are likely not new.
However, we have been unable to find an original work, so we provide proofs for completeness.

Let $(K, d)$ be a precompact metric space, and let $\mathcal{N}(K, \delta)$ denote the minimal number of
balls of radius $\delta$ necessary to cover $K$. Let $x_0 \in K$ be a fixed point, and let 
$\mathcal{F}$ denote the set of $1$-Lipschitz functions over $K$ vanishing at $x_0$.
This is also a precompact space when endowed with the metric of uniform convergence.
We denote by $\mathcal{N}(\mathcal{F}, \delta)$ the minimal number of balls of radius $\delta$ necessary to cover
$\mathcal{F}$. Finally, we set $R = \max_{x \in K} d(x, x_0)$.

Our first estimate is a fairly crude one.

\begin{prop}
We have

\begin{equation*}
 \mathcal{N}(\mathcal{F}, \eps) \leq \left( 2 + 2 \lfloor \frac{3 R}{\eps} \rfloor \right)^{\mathcal{N}(K, \frac{\eps}{3})}.
\end{equation*}

\end{prop}

\begin{proof}
For simplicity, write $n = \mathcal{N}(K, \eps)$.
 Let $x_1, \ldots, x_n$ be the centers of a set of balls covering $K$.
For any $f \in \mathcal{F}$ and $1 \leq i \leq n$, we have

\begin{equation*}
 |f(x_i)| = |f(x_i) - f(x_0)| \leq R.
\end{equation*}

For any $n$-uple of integers $k = (k_1, \ldots, k_n)$ such that
 $- \lfloor \frac{R}{\eps} \rfloor - 1 \leq k_i \leq \lfloor \frac{R}{\eps} \rfloor$, 
$1 \leq i \leq n$,
choose a function $f_k \in \mathcal{F}$ such that $k_i \eps \leq f_k(x_i) \leq (k_i + 1) \eps$ if there exists one.

Consider $f \in \mathcal{F}$. 
Let $l_i = \lfloor \frac{f(x_i)}{\eps} \rfloor$ and $l = (l_1, \ldots, l_n)$.
Then the function $f_l$ defined above exists and $|f(x_i) - f_l(x_i)| \leq \eps$ for $1 \leq i \leq n$. But then for any $x \in K$ there exists $i$, $1 \leq i \leq n$,
 such that $x \in B(x_i, \eps)$, and thus

\begin{equation*}
 |f(x) - f_l(x) | \leq |f(x) - f(x_i)| + |f(x_i) - f_l(x_i)| + |f_l(x_i) - f_l(x)| \leq 3 \eps.
\end{equation*}

This implies that $\mathcal{F}$ is covered by the balls of center $f_k$ and radius $3 \eps$. As there are at most
$(2 + 2 \lfloor \frac{R}{\eps} \rfloor)^n$ choices for $k$, this ends the proof.

\end{proof}

However, this bound is quite weak : as one can see by considering the case of a segment, for most choices of a $n$-uple, there
will not exist a function in $\mathcal{F}$ satisfying the requirements in the proof.
With the extra assumption that $K$ is connected, we can get a more refined result.

\begin{prop} \label{prop_covering_number_lip}
 If $K$ is connected, then 

\begin{equation} \label{covering_number_lip_2}
 \mathcal{N}( \mathcal{F}, \eps) \leq \left( 2 + 2 \lfloor \frac{4 R}{\eps} \rfloor \right) \, 2^{ \displaystyle \mathcal{N}(K, \frac{\eps}{16}) }.
\end{equation}

\end{prop}

\begin{remark}
 The simple idea in this proposition is first to bring the problem to a discrete metric space
 (graph), and then to bound the number of
Lipschitz functions on this graph by the number of Lipschitz functions on a spanning tree of the graph.
\end{remark}

\begin{proof}
 
In the following, we will denote $n = \mathcal{N}(K, \eps)$ for simplicity.
Let $x_i$, $1 \leq i \leq n$ be the centers of a set of $n$ balls 
$B_1, \ldots B_n$ covering $K$. Consider the graph $G$ 
built on the $n$ vertices $a_1, \ldots, a_n$,
where vertices $a_i$ and $a_j$ are connected if and only if
$i \neq j$ and the balls $B_i$ and $B_j$ have a non-empty intersection.

\begin{lemma}
The graph $G$ is connected. Moreover, there exists a subgraph $G'$ 
with the same set of vertices and whose edges are edges of $G$, which is a tree.
\end{lemma}

\begin{proof}
 Suppose that $G$ were not connected . Upon exchanging the labels of the balls, there would
exist $k$, $1 \leq k < n$, such that for $i \leq k < j$ the
balls $B_i$ and $B_j$ have empty intersection.
But then $K$ would be equal to the disjoint reunion of the sets
$\bigcup_{i = 1}^k B_i$ and $\bigcup_{j = k+1}^n B_j$, which are both closed and
non-empty, contradicting the connectedness of $K$.

The second part of the claim is obtained by an easy induction on the size of the graph.
\end{proof}

Introduce the set $\mathcal{A}$
of functions $g : \{a_1, \ldots, a_n \} \rightarrow \mathbb{R}$ such that
$g(a_1) = 0$ and $|g(a_i) - g(a_j)| = 4 \eps$ whenever $a_i$ and $a_j$ are connected in $G'$.
Using the fact that $G'$ is a tree, it is easy to see that $\mathcal{A}$ contains at most
$2^n$ elements.

Define a partition of $K$ by setting $C_1 = B_1$, 
$C_2 = B_2 \backslash C_1$, $\ldots$, $ C_n = B_n \backslash C_{n-1}$ (remark that none of the $C_i$ is empty since
the $B_i$ are supposed to constitute a minimal covering).
Also fix for each $i$, $1 \leq i \leq n$, a point $y_i \in C_i$ (choosing $y_1 = x_1$). Notice that
$C_i$ is included in the ball of center $y_i$ and radius $2 \eps$, and that 
$d(y_i, y_j) \leq 4 \eps$ whenever $a_i$ and $a_j$ are connected in $G$ (and therefore in $G'$).

To every $1$-Lipschitz function $f : K \rightarrow \R$ we associate $T(f) : \{a_1, \ldots, a_n \} \rightarrow \R$
defined by $T(f) (a_i) = f(y_i)$. For any $x \in K$, and $f_1$, $f_2 \in \mathcal{F}$, we have the following :

\begin{eqnarray*}
 |f_1(x) - f_2(x) | & \leq & |f_1(x) - f_1(y_i) | + |f_1(y_i) - f_2(y_i) | + |f_2(y_i) - f_2(x)| \\
{} & \leq & 4 \eps + \|T(f_1) - T(f_2)\|_{ \ell_\infty(G') }
\end{eqnarray*}

where $i$ is such that $x \in C_i$.
We now make the following claim : 

\begin{lemma}  \label{lip_approx}
 For every $1$-Lipschitz function $f : K \rightarrow \R$ such that
$f(y_1) = 0$, there exists $g \in \mathcal{A}$ such that $\| T(f) - g \|_{\ell_\infty(G')} \leq 4 \eps$.
\end{lemma}

Assume for the moment that this holds. As there are at most $2^n$ functions in $\mathcal{A}$, it is possible to choose
at most $2^n$ $1$-Lipschitz functions $f_1, \ldots, f_{2^n}$ vanishing at $x_1$ such that for any 
$1$-Lipschitz function f vanishing at $x_1$ there exists $f_i$ such that $|T(f) - T(f_i)| \leq 8 \eps$.
Using the inequality above, this implies that the balls of center $f_i$ and radius $12 \eps$ for the uniform distance
 cover the set of $1$-
Lipschitz functions vanishing at $x_1$.

Finally, consider $f \in \mathcal{F}$. We may write

\begin{equation*}
 f = f - f(x_1) + f(x_1)
\end{equation*}

and observe that on the one hand, $f - f(x_1)$ is a $1$-Lipschitz function vanishing at $x_1$, and that on the other hand,
 $|f(x_1)| \leq R$.
Thus the set $\mathcal{F}$ is covered by the balls of center $f_i + 4 k \eps$ and radius $16 \eps$,
where $- \lfloor \frac{R}{ 4 \eps} \rfloor - 1 \leq k \leq \lfloor \frac{R}{4 \eps} \rfloor$.
There are at most $(2 + 2 \lfloor \frac{R}{4 \eps} \rfloor) 2^n$ such balls, which proves the desired result.

\end{proof}

We now prove Lemma \ref{lip_approx}.
\begin{proof}
 
Let us use induction again. If $K = B_1$ then $T(f) = 0$ and the property is straightforward.
Now if $K = C_1 \cup \ldots \cup C_n$, we may assume without loss of generality that $a_n$ is a leaf in $G'$, that
is a vertex with exactly one neighbor, and that it is connected to $a_{n-1}$.
By hypothesis there exists $\tilde{g}: \{ a_1, \ldots, a_{n-1} \} \rightarrow \R$
 such that $|\tilde{g}(a_i) - \tilde{g}(a_j)| = 4 \eps$ 
whenever $a_i$ and $a_j$ are connected in $G'$, and $|\tilde{g}(a_i) - f(a_i)| \leq 4 \eps$ for $1 \leq i < n$.
Set $g = \tilde{g}$ on $ \{a_1, \ldots, a_{n-1} \}$, and
\begin{itemize}
 \item $g(a_n) = g(a_{n-1}) + 4 \eps$ if $f(y_n) - g(a_{n-1}) < 0$,
 \item $g(a_n) =  g(a_{n-1}) - 4 \eps$ otherwise.
\end{itemize}

Since
\begin{equation*}
|f(y_n) - g(a_{n-1}) | \leq |f(y_n) - f(y_{n-1})| + |f(y_{n-1}) - g(a_{n-1}) | \leq 8 \eps
\end{equation*}
it is easily checked that $|f(y_n) - g(a_n) | \leq 4 \eps$. The function $g$ belongs to $\mathcal{A}$ and our claim is proved.

\end{proof}

%% file: appendix_2.tex
\section{H\"{o}lder moments of stochastic processes}

We quote the following result from Revuz and Yor's book \cite{revuz1999continuous} (actually
the value of the constant is not given in their statement but is easily tracked in the proof).

\begin{theorem} \label{kolmogorov_theorem}
 Let $X_t$, $t \in [0, 1]$ be a Banach-valued process such that there exist $\gamma, \varepsilon, c > 0$ with

\begin{equation*}
 \mathbb{E} \left[|X_t-X_s| ^\gamma \right] \leq c |t-s|^{1+ \varepsilon},
\end{equation*}

then there exists a modification $\tilde{X}$ of $X$ such that

\begin{equation*}
 \mathbb{E} \left[ \left( \sup_{s \neq t} \frac{|\tilde{X}_t - \tilde{X}_s|}{|t-s|^\alpha} \right)^\gamma \right]^{1/\gamma} \leq 2^{1 + \alpha} (2c)^{1/\gamma} \frac{1}{1 - 2^{\alpha - \varepsilon/\gamma}}
\end{equation*}

for all $0 \leq \alpha < \varepsilon/\gamma$.

\end{theorem}

\begin{corollary} \label{holder_moments_sbm}
Let $(B_t)_{0 \leq t \leq 1}$ denote the standard Brownian motion on $[0, 1]$.
 Let  $m_\alpha = \mathbb{E} \sup_t \| B_t \|_\alpha$ and $s^2_\alpha = \mathbb{E} (\sup_t \| B_t \|_\alpha)^2$, then $m_\alpha$ and $s_\alpha$ are bounded by

\begin{equation*}
 C_\alpha = 2^{1+ \alpha}\frac{2^{(1-2 \alpha)/4}}{1 - 2^{(2 \alpha-1)/4}} \|Z\|_{4/(1-2\alpha)}
\end{equation*}

where $\|Z\|_p$ denotes the $L_p$ norm of a $\mathcal{N}(0, 1)$ variable $Z$.

\end{corollary}

\begin{proof}
Since the increments of the Brownian motion are Gaussian, we have for every $p > 0$

\begin{equation*}
 \mathbb{E}[ |B_t - B_s|^{2p} ] = K_p |t-s|^p
\end{equation*}

with $K_p = \sqrt{2 \pi}^{-1} \int_{- \infty}^{+\infty} |x|^{2p} e^{-x^2/2} dx$.
Choose $p$ such that $\alpha < (p-1)/2p$, then

\begin{equation*}
 \left( \mathbb{E} \|X\|_\alpha^{2p} \right)^{1/2p} \leq \frac{2^{1+ \alpha}}{1 - 2^{\alpha - 1/2 + 1/2p}} (2 K_p)^{1/2p}.
\end{equation*}

A suitable choice is $1/p = 1/2 - \alpha$, and the right-hand side becomes

\begin{equation*}
 C_\alpha = \frac{2^{1+ \alpha}}{1 - 2^{(\alpha - 1/2)/2}} (2 G_\alpha)^{(1/2 - \alpha)/2}
\end{equation*}

with $G_\alpha = \sqrt{2 \pi}^{-1} \int_{- \infty}^{+\infty} |x|^{4/(1 - 2 \alpha)} e^{-x^2/2} dx$.
By H\"{o}lder's inequality, the result follows.

\end{proof}

\begin{corollary} \label{corollary_holder_sde}

Let $X_t$ be the solution on $[0, T]$ of

\begin{equation*}
 d X_t = \sigma(X_t) dB_t + b(X_t) dt
\end{equation*}
 
with $\sigma, b : \R \rightarrow \R$ locally Lipschitz and satisfying the following hypotheses :

\begin{itemize}
 \item  $\sup_x | \sqrt{ \text{tr}\sigma(x)^t\sigma(x)} | \leq A$,
 \item $\sup_x |b(x)| \leq B$.
\end{itemize}

Then for $\alpha < 1/2$, for $p$ such that $\alpha < (p-1)/2p$, 
there exists $C < + \infty$ depending explicitely on $A$, $B$, $T$, $\alpha$ ,$p$
such that

\begin{equation*}
 \mathbb{E}\|X\|_\alpha^p \leq C.
\end{equation*}

\end{corollary}

\begin{proof}

We first apply It\^{o}'s formula to the function $|X_t - X_s|^2$ : this yields

\begin{equation*}
 \mathbb{E}|X_t - X_s|^2 \leq 2 B \int_s^t \mathbb{E}|X_u - X_s| du + A |t-s|.
\end{equation*}

Using the elementary inequality $x \leq 1/2(1 + x^2)$, we get

\begin{equation*}
 \mathbb{E}|X_t - X_s|^2 \leq B \int_s^t \mathbb{E}|X_u - X_s|^2 du + (A + B)|t-s|.
\end{equation*}

Gronwall's lemma entails

\begin{equation*}
 \mathbb{E}|X_t - X_s|^2 \leq (A + B) e^{B T} |t -s|
\end{equation*}

Likewise, applying It\^{o}'s formula to $|X_t - X_s|^4$, we get

\begin{align*}
  \mathbb{E}|X_t - X_s|^4 & \leq 4 B \int_s^t \mathbb{E}|X_u - X_s|^3 ds + 6A \int_s^t \mathbb{E}|X_u-X_s|^2 du \\
{} & \leq (6A+2B) \int_s^t \mathbb{E}|X_u-X_s|^2 du + 2B \int_s^t \mathbb{E}|X_u-X_s|^4 du \\
{} &\leq \frac{1}{2}(6A+2B)(A+B)e^{BT}|t-s|^2 + 2B\int_s^t \mathbb{E}|X_u-X_s|^4 du
\end{align*}

and by Gronwall's lemma $\mathbb{E}|X_t - X_s|^4 \leq \frac{1}{2}(6A+2B)(A+B)e^{3BT}|t-s|^2$.
By an easy recurrence, following the above, one may show that

\begin{equation*}
 \mathbb{E}|X_t - X_s|^{2p} \leq C(A, B, T, p) |t-s|^p.
\end{equation*}

To conclude it suffices to call on Theorem \ref{kolmogorov_theorem}.

\end{proof}

%% file: appendix_3.tex
\section{Transportation inequalities for Gaussian measures on a Banach space} \label{appendix_gaussian}

\begin{lemma} \label{lemma_gaussian_tail}
 Let $(E, \mu)$ be a Gaussian Banach space, and define $m = \int \|x\| \mu(dx)$. Also let $\sigma^2$ denote the weak variance of $\mu$.
The tail of $\mu$ is bounded as follows : for all $R \geq 0$,

\begin{equation*}
 \mu \{x \in E, \; \|x\| \geq m + R \} \leq e^{- R ^2/2\sigma^2 }.
\end{equation*}
\end{lemma}

\bigskip

Finally we collect some (loose) results on the Wiener measure on the Banach space $(C([0, 1], \R), \|.\|_\infty)$.

\begin{lemma}
 The Wiener measure satisfies a $\bt_2(8)$ inequality (and therefore a $\bt_1(8)$ inequality).
\end{lemma}

\begin{proof}
The Wiener measure satisfies the $\bt_2(2 \sigma^2)$ inequality,
where 

\begin{equation*}
 \sigma^2 = \sup_{\mu} \mathbb{E} ( \int_0^1 B_s d \mu(s))^2
\end{equation*}

and the supremum runs over all Radon measures on $[0, 1]$ with total variation bounded by $1$.
Note that the weak variance $\sigma^2$ is bounded by the variance $s^2$ defined as
$s^2 = \mathbb{E} (\sup_t |B_t|) ^2$
 (here
and hereafter $\sup_t |B_t|$ refers to the supremum on $[0, 1]$).
In turn we can give a (quite crude) bound on $s$ :
write $\sup_t |B_t| \leq \sup_t B_t - \inf_t B_t$, thus
$(\sup_t |B_t|)^2 \leq ( \sup_t B_t - \inf_t B_t)^2 \leq 2 (\sup_t B_t)^2 + 2 (- \inf_t B_t)^2$.
Remember the well-known fact that $\sup_t B_t$, $- \inf_t B_t$ and $|B_1|$ have the same law, so that

\begin{equation*}
 \mathbb{E}( \sup_t |B_t|)^2 \leq 4 \mathbb{E} |B_1|^2 = 4.
\end{equation*}

\end{proof}

\begin{lemma} \label{lemma_small_exp_moment_wiener}
 Let $\gamma$ denote the Wiener measure. For $a = \sqrt{2 \log 2}/3$, we have
\begin{equation*}
\int e^{a\|x\|_\infty} \gamma(dx) \leq 2
\end{equation*}.
\end{lemma}

\begin{proof}
 
We have

\begin{align*}
 \int e^{a\|x\|_\infty} \gamma(dx) & = \int_0^{ + \infty} \pr ( e^{a \|x\|_\infty} \geq t) d t \\
{} & = \int_0^{+ \infty} \pr(  \|x\|_\infty \geq u) a e^{a u} du \\
{} & =  \int_0^{+ \infty} \pr(  \tau_u \leq 1) a e^{a u} du
\end{align*}

where $\tau_u$ is the stopping time $\inf \{t, |B_t| = u \}$. It is a simple exercise to compute

\begin{equation*}
 \mathbb{E} e^{- \lambda^2 \tau_u / 2} = 1/ \cosh (\lambda u) \leq 2 e^{- \lambda u}.
\end{equation*}

This yields

\begin{equation*}
  \int e^{a\|x\|_\infty} \gamma(dx) \leq 2 a e^{\lambda^2/2}\int_0^{+ \infty} e^{(a - \lambda) u} du = \frac{2 a e^{ \lambda^2/2}}{\lambda - a}.
\end{equation*}

We can choose $\lambda = 3 a$ to get  $\int e^{a\|x\|_\infty} \gamma(dx) \leq e^{9a^2/2}$. In turn it implies the desired result.

\end{proof}